\documentclass[10pt,twoside,a4paper]{amsart}

\usepackage[T1]{fontenc}
\usepackage{amsmath}
\usepackage{amssymb}
\usepackage{amsthm}
\usepackage{thmtools}
\usepackage{enumitem}
\usepackage{mathrsfs}
\usepackage{bbm}
\usepackage{xcolor}
\usepackage[left=3cm,right=3cm,top=3.5cm,bottom=3cm]{geometry}
\usepackage[colorlinks=true, linktocpage=true, linkcolor=red!70!black, citecolor=green!50!black]{hyperref}

\newcommand{\bb}{\mathbb}

\renewcommand{\th}{\textsuperscript{th}\ }
\providecommand{\eps}{\varepsilon}

\newcommand{\norm}[1]{\left\lVert#1\right\rVert}

\providecommand{\tensor}{\otimes}

\renewcommand{\d}{\mathrm{d}}
\newcommand{\D}{\mathrm{D}}
\AtBeginDocument{

}

\newcommand{\pdeproblem}[6]{
  \setlength{\arraycolsep}{0pt}
  \renewcommand{\arraystretch}{1.2}
  \left\{\begin{array}{r @{\ } l @{\ } l}
    #1 &= #2 &\text{ in } #3 \\
    #4 &= #5 &\text{ on } #6 \\
  \end{array}\right.
}

\DeclareMathOperator*{\esssup}{ess\ sup}

\DeclareMathOperator{\dist}{dist}

\DeclareMathOperator{\Sym}{Sym}
\DeclareMathOperator{\loc}{loc}
\DeclareMathOperator{\BV}{BV}

\DeclareMathOperator{\LL}{L}
\DeclareMathOperator{\CC}{C}
\DeclareMathOperator{\WW}{W}

\newcommand{\B}{\mathrm{B}}

\newcommand{\twopartdef}[4] { \left\{ \begin{array}{ll} #1 &  #2 \\ #3 &  #4 \end{array} \right.}

\def\Xint#1{\mathchoice
  {\XXint\displaystyle\textstyle{#1}}%
  {\XXint\textstyle\scriptstyle{#1}}%
  {\XXint\scriptstyle\scriptscriptstyle{#1}}%
  {\XXint\scriptscriptstyle\scriptscriptstyle{#1}}%
  \!\int}
\def\XXint#1#2#3{\setbox0=\hbox{$#1{#2#3}{\int}$}\vcenter{\hbox{$#2#3$}}\kern-.5\wd0}

\def\dashint{\Xint-}
\theoremstyle{definition}
\newtheorem{thm}{Theorem}[section]
\newtheorem*{thm*}{Theorem}

\newtheorem{defn}[thm]{Definition}

\newtheorem{lem}[thm]{Lemma}

\newtheorem{prop}[thm]{Proposition}

\newtheorem{hyp}[thm]{Hypotheses}

\newtheorem{rem}[thm]{Remark}

\newcommand*{\doi}[1]{\href{https://doi.org/#1}{\nolinkurl{#1}}}
\newcommand*{\arxiv}[1]{\href{https://arxiv.org/abs/#1}{\nolinkurl{#1}}}

\numberwithin{equation}{section}

\begin{document}

\title[Higher quasiconvex regularity with Orlicz growth]{Partial regularity for minima of higher-order quasiconvex integrands with natural Orlicz growth}

\author[C.~Irving]{Christopher Irving}
\thanks{\textit{Funding}: The author was supported by the EPSRC [EP/L015811/1]}
\address{Faculty of Mathematics\\ Technical University Dortmund\\ Vogelpothsweg 87\\ 44227\\ Dortmund\\ Germany}
\email{christopher.irving@tu-dortmund.de}
\date{\today}

\maketitle

\begin{abstract}
A partial regularity theorem is presented for minimisers of $k^{\mathrm{th}}$-order functionals subject to a quasiconvexity and general growth condition. We will assume a natural growth condition governed by an $N$-function satisfying the $\Delta_2$ and $\nabla_2$ conditions, assuming no quantitative estimates on the second derivative of the integrand; this is new even in the $k=1$ case. These results will also be extended to the case of strong local minimisers.
\end{abstract}

\section{Introduction}

In this paper we will investigate the regularity of minimisers of functionals of the form
\begin{equation}\label{eq:functional_autonomous}
  \mathcal F(w) = \int_{\Omega} F(\nabla^k w) \,\d x,
\end{equation} 
where $\Omega \subset \bb R^n$ is a bounded domain and $w \colon \Omega \to \bb R^N$ is a vector-valued mapping. We will assume $F$ satisfies a suitable \emph{higher-order quasiconvexity condition} in the sense of \textsc{Morrey} \cite{article:Morrey52} for $k=1$ and \textsc{Meyers} \cite{article:Meyers65} for $k\geq2,$ and a general growth condition governed by an $N$-function $\varphi$ which we specify below. If $F$ is sufficiently regular, it is known that minimisers $u$ of (\ref{eq:functional_autonomous}) are \emph{$F$-extremal,} that is they satisfy the associated Euler-Lagrange system
\begin{equation}\label{eq:F_extremal}
  (-1)^k \nabla^k :  F'(\nabla^k u) := (-1)^k\sum_{\lvert\alpha\rvert=k} \D^{\alpha} \left(F'(\nabla^ku) : \mathrm e^{\alpha}\right) = 0
\end{equation} 
in $\Omega,$ using the notation detailed in Section \ref{sec:notation}.

A striking feature of vector-valued problems is that minimisers need not be regular in general (see for instance \cite{article:degiorgi68,article:Mazya69,article:Necas77,article:SverakYan02,article:MooneySavin16}), however we can still hope for \emph{partial regularity results}, which assert that, away from a small singular set, minimisers $u$ are as regular as the data allows. The first partial regularity result in the quasiconvex setting was established by \textsc{Evans} \cite{article:Evans86} which has been extended significantly since; see \cite{article:FuscoHutchinson85,article:GiaquintaModica86,article:AcerbiFusco87,article:EvansGariepy87,article:CarozzaFuscoMingione,article:AcerbiMingione01,article:Kronz02,article:KristensenTaheri03,article:Dieningetal12,article:Bogelein12,article:GmeinederKristensen19,article:DeFilippis21,article:DeFilippisStroffolini2023} for a non-exhaustive list.

In the higher order case, the regularity theory for strongly $k$-quasiconvex integrands has been studied for instance by \textsc{Guidorzi} \cite{article:Guidorzi00}, \textsc{Kronz} \cite{article:Kronz02}, and \textsc{Schemm} \cite{article:Schemm11}. Note the results of \cite{article:DalMasoEtAl04,article:Cagnetti11} show that higher-order quasiconvexity reduces to ordinary quasiconvexity under quantitative Lipschitz bounds on $F',$ and while these results cannot be directly applied, we expect the higher order case to be essentially the same as the $k=1$ setting.

We will mention that in the quasiconvex setting, some kind of minimising condition is essential to obtain regularity in general. This is illustrated in the seminal work of \textsc{M\"uller \& \v{S}ver\'ak} \cite{article:MullerSverak03} where Lipschitz but nowhere $\CC^1$ solutions to the associated Euler-Lagrange system are constructed, and this has been extended by \textsc{Kristensen \& Taheri} \cite{article:KristensenTaheri03} and \textsc{Sz\'ekelyhidi} \cite{article:Szekelyhidi04} to weak local minimisers and strongly polyconvex integrands respectively. As illustrated in \cite{article:KristensenTaheri03} one can infer partial regularity for strong local minimisers however, which we discuss in Section \ref{sec:local_min}.

\subsection{Hypotheses and main results}\label{sec:intro_results}

We will consider higher-order integrands satisfying a natural growth condition governed by an $N$-function, and a strict quasiconvexity condition which matches said growth bound. We refer the reader to Sections \ref{sec:notation}, \ref{sec:Nfunctions} for the precise definitions and conventions we use.

\begin{hyp}\label{hyp:generalgrowth_F}

  Let $F \colon \bb M_k \to \bb R$ with $n \geq 2,$ $N,k\geq 1$ satisfy the following.
 
\begin{enumerate}[label=(H\arabic*)]
  \setcounter{enumi}{-1}

  \item $F$ is $\CC^2$ regular.

  \item There exists $K \geq 0$ and an $N$-function $\varphi$ satisfying the $\Delta_2$ and $\nabla_2$ conditions (as defined in Section \ref{sec:Nfunctions}) such that
    \begin{equation*}\label{eq:F_growth}
      \lvert F(z)\rvert \leq K \left( 1 + \varphi(\lvert z\rvert) \right) 
    \end{equation*} 
    for all $z \in \bb M_k.$ By rescaling $\varphi, K$ if necessary, we will assume $\varphi(1)=1.$

  \item $F$ is \emph{strictly $\WW^{k,\varphi}$-quasiconvex} in the sense that there is $\nu>0$ such that
    \begin{equation*}\label{eq:F_quasiconvex}
      \int_{\bb R^n} F(z_0 + \nabla^k \xi) - F(z_0) \,d x \geq \nu \int_{\bb R^n}  \varphi_{1+\lvert z_0\rvert}(\lvert\nabla^k \xi\rvert) \,\d x
    \end{equation*} 
    for all $z_0 \in \bb M_k$ and $\xi \in \CC^{\infty}_c(\bb R^n),$ where $\varphi_a$ is defined in (\ref{eq:varphi_shifted}).
\end{enumerate}
 
\end{hyp}

For integrands of this type, we say $u \in \WW^{k,\varphi}(\Omega,\bb R^N)$ is a \emph{minimiser} of (\ref{eq:functional_autonomous}) if for any $\xi \in \WW^{k,\varphi}_0(\Omega,\bb R^N)$ we have
\begin{equation}
  \mathcal F(u) \leq \mathcal F(u+\xi).
\end{equation} 
Note that it suffices to verify this for $\xi \in \CC^{\infty}_c(\Omega,\bb R^N),$ from which the above follows by density (using both \ref{eq:F_growth} and \ref{eq:F_quasiconvex}).

We will point out that, while the partial regularity for quasiconvex integrands satisfying a general growth condition have been considered by \textsc{Diening, Lengeler, Stroffolini, \& Verde} \cite{article:Dieningetal12} for the $k=1$ case, the authors assume a \emph{controlled growth condition} where quantitative estimates are assumed on the second derivatives $F''.$ We will relax this condition and establish the following.

\begin{thm}[$\eps$-regularity theorem]\label{thm:main_epsreg}
  Let $F$ satisfy Hypotheses \ref{hyp:generalgrowth_F} and let $M>0,$ $\alpha \in (0,1)$ be given. Then there exists $\eps >0$ such that if $u \in \WW^{k,\varphi}(\Omega,\bb R^N)$ minimises (\ref{eq:functional_autonomous}) where $\Omega \subset \bb R^n$ is a bounded domain, and $\B_R(x_0) \subset \Omega$ is such that
  \begin{equation}
    \lvert(\nabla^k u)_{\B_R(x_0)}\rvert \leq M, \quad \dashint_{\B_R(x_0)} \varphi_{1+M}\left( \lvert\nabla^k u - (\nabla^k u)_{\B_R(x_0)}\rvert \right) \,\d x \leq \eps,
  \end{equation} 
  then we have $u$ is of class $\CC^{k,\alpha}$ in $\B_{R/2}(x_0).$
\end{thm}

Our proof largely hinges on a suitable Caccioppoli inequality and a harmonic approximation argument, which can be traced back to the works of \textsc{Giusti \& Miranda} \cite{article:GiustiMiranda68} and \textsc{Morrey} \cite{article:Morrey68} in the general variational context, which were inspired by prior works in the geometric setting. The Caccioppoli inequality will be based on of the version in \cite{article:Evans86}, which was adapted to the general growth setting in \cite{article:Dieningetal12}. For the harmonic approximation argument we will adapt a recent approach of \textsc{Gmeineder \& Kristensen} \cite{article:GmeinederKristensen19}, which uses a duality argument to directly estimate the excess in the approximation. This approach has also been surveyed in the $p$-growth setting by \textsc{B\"arlin, Gmeineder, Kristensen} and the author in \cite{article:BarlinEtAl24}.

To apply this harmonic approximation argument from \cite{article:GmeinederKristensen19} in the general growth setting, we will need to additionally use a version of the Lipschitz truncation lemma of \textsc{Acerbi \& Fusco} \cite{article:AcerbiFusco84,article:AcerbiFusco87}. This approach parallels the $\mathcal A$-harmonic approximation argument which has appeared in for instance \cite{article:DuzaarGrotowski00,article:DuzaarSteffen02,article:Kronz02,article:DuzaarMingione04,article:DieningEtAl12a}, which traces back to arguments from geometric problems and can be found in texts of \textsc{Simon} \cite{book:Simon83,book:Simon96}. In \cite{article:Dieningetal12} a direct proof of a suitable $\mathcal A$-harmonic approximation is proven by applying a Lipschitz truncation to a suitable dual problem, which closely mirrors the strategy we adopt.

Once the above $\eps$-regularity theorem is established, the following partial regularity theorem follows by standard means.

\begin{thm}[Partial regularity of minimisers]\label{thm:partial_reg}
  Let $F$ satisfy Hypotheses \ref{hyp:generalgrowth_F}, and $u \in \WW^{k,\varphi}(\Omega,\bb R^N)$ be a minimiser of (\ref{eq:functional_autonomous}), where $\Omega \subset \bb R^N$ is a bounded domain. Then there exists an open subset $\Omega_0\subset\Omega$ of full measure such that $u$ is $\CC^{k,\alpha}$ in $\Omega_0$ for each $\alpha \in (0,1).$ Moreover we have $\Omega_0 = \Omega \setminus \left(\Sigma_1 \cup \Sigma_2\right),$ where
  \begin{align}
    \Sigma_1 &= \left\{ x \in \Omega \colon \limsup_{r \to 0} \dashint_{\B_r(x)} \lvert\nabla u\rvert\,\d x = \infty \right\} \\
    \Sigma_2 &= \left\{ x \in \Omega \colon \limsup_{r \to 0} \dashint_{\B_r(x)} \varphi_{1+\lvert(\nabla u)_{\B_r(x)}\rvert}\left(\lvert \nabla u - (\nabla u)_{\B_r(x)}\rvert\right)\,\d x > 0 \right\}.
  \end{align}
\end{thm}

In light of the partial regularity results in the linear growth setting \cite{article:GmeinederKristensen19}, a natural question is whether the $\nabla_2$-condition is really necessary. In fact the results in \cite{article:GmeinederKristensen24} (see also \cite{article:GmeinederKristensen19a}) implies that partial regularity holds if we have merely the growth condition $\varphi(t) \leq C t^q$ with $q < \frac{n}{n-1},$ however the general case remains open, which we wish to address in future work.

\section{Preliminaries}

\subsection{Notation}\label{sec:notation}

We will briefly fix some notation which will be used throughout the text. We denote by $\bb M_k = \Sym_k(\bb R^n,\bb R^N)$ the space of symmetric $k$-linear maps $(\bb R^n)^k \to \bb R^N,$ which is equipped with the inner product $z:w = \sum_{\lvert\alpha\rvert=k} z(\mathrm e^{\alpha})\cdot w(\mathrm e^{\alpha})$ for $z,w \in \bb M_k,$ taking tensor powers of the standard orthonormal basis $\{\mathrm e_i\}_{i=1}^n$  of $\bb R^n.$ The associated norm is denoted $\lvert z\rvert = \sqrt{z : z}.$ Note that we identify $\bb M_0 = \mathbb R^N$ as vectors in $\mathbb R^N$ and $\bb M_1 = \bb R^{Nn}$ as the space of $N \times n$ matrices. If $u \colon \Omega \to \bb R^N$ is $k$-times continuously differentiable where $\Omega \subset \bb R^n$ is an open set, we denote the partial derivatives of $u$ by $\D^{\alpha}u$ using multi-index notation, and its $k$\th order gradient by $\nabla^k u \colon \Omega \to \mathbb M_k$, given by $\nabla^ku(\mathrm e^{\alpha}) = \D^{\alpha} u$. The same notation will be used for weak derivatives.

We will equip $\bb R^n$ with the Lebesgue measure $\mathcal L^n,$ and if $A \subset \bb R^n$ is non-empty and open such that $0<\mathcal L^n(A)<\infty$, for any $f \in \LL^1(A,\bb V)$ with $(\bb V,\lvert\,\cdot\,\rvert)$ a finite-dimensional real vector space, we define
\begin{equation}
  (f)_A := \dashint_A f \,\d x := \frac1{\mathcal L^n(A)} \int_A f \,\d x.
\end{equation}
We also denote by a $\B_R(x_0)$ the open ball in $\bb R^n$ centred at $x_0$ with radius $R.$

For a differentiable map $F \colon \bb M_k \to \bb R$ we define its derivative $F' \colon \bb M_k \to \bb M_k$ as 
\begin{equation}
  F'(z)w = \left.\frac{\d}{\d t}\right|_{t=0} F(z+tw),
\end{equation}
and if $F$ is $\CC^2,$ its second derivative $F''(z)$ will be a linear map $\bb M_k \to \bb M_k$ satisfying
\begin{equation}
  F''(z)v : w = \left.\frac{\d}{\d t}\right|_{t=0} F'(z+tv) w.
\end{equation}
This can be viewed as a symmetric bilinear form on $\bb M_k.$

Additionally $C, C_1, C_2,\dots$ will denote constants which may change from line to line, and if not specified will depend only on the parameters the resulting estimate depends on. We may also write $C_{\alpha,\beta,\dots}$ to emphasise the dependence on certain parameters. We also write $A \sim B$ if there exists constants $C_1,C_2>0$ such that $C_1A \leq B \leq C_2A.$

\subsection{$N$-functions}\label{sec:Nfunctions}

We will define the scales of growth we are interested in, and record some basic properties. These results can be found for instance in \cite{article:KrasnoselskiiRutickii61,article:RaoRen91,book:AdamsFournier03}.

We say $\varphi \colon [0,\infty) \to [0,\infty)$ is an \emph{$N$-function} if $\varphi$ is increasing, continuous and convex such that $\varphi(t)=0$ if and only if $t=0$ and we have the limits
\begin{equation}\label{eq:phi_taillimits}
  \lim_{t \to 0} \frac{\varphi(t)}{t} = 0, \quad \lim_{t \to \infty} \frac{\varphi(t)}t = \infty
\end{equation} 
hold true.
Given such a $\varphi$, define its \emph{conjugate function} by
\begin{equation}\label{eq:conjugate_defn}
  \varphi^*(s) = \int_0^s (\varphi')^{-1}(\sigma)\,\d \sigma = \int_0^s\inf\{\tau > 0 \colon \varphi'(\tau) > \sigma\}\,\d \sigma.
\end{equation} 

We say an $N$-function $\varphi$ satisfies the \emph{$\Delta_2$-condition} if there is $C \geq 1$ for which $\varphi(2t) \leq C\varphi(t)$ holds for all $t\geq 0,$ and the least such $C$ will be denoted $\Delta_2(\varphi).$ We also say $\varphi$ satisfies the \emph{$\nabla_2$-condition} if its conjugate $\varphi^*$ satisfies the $\Delta_2$-condition, and write $\nabla_2(\varphi) = \Delta_2(\varphi^*).$ We will use the notation $\varphi \in \Delta_2,$ $\varphi \in \nabla_2$ to denote $N$-functions satisfying the $\Delta_2$ and $\nabla_2$ conditions respectively, and write $\varphi \in \Delta_2\cap\nabla_2$ if both are satisfied.

Note that if $\varphi \in \Delta_2,$ then there is $p>1$ such that $\varphi(st) \leq s^{p} \varphi(t)$ for all $t > 0, s>1.$ The minimal such $p$ will be denoted by $p_{\varphi}.$ 

\begin{lem}
  The following inequalities hold for an $N$-function $\varphi.$

  \begin{enumerate}[label=(\alph*)]
    \item (Young's inequality) For $t,s \geq 0$ we have
      \begin{equation*}\label{eq:young_inequality}
        ts \leq \varphi(t) + \varphi^*(s),
      \end{equation*} 
      with equality if and only if $s= \varphi'(t)$ or $t=(\varphi^*)'(s).$

    \item For any $t \geq 0$ we have
      \begin{equation*}\label{eq:varphi_inverse}
        t \leq \varphi(t)^{-1}(\varphi^*)^{-1}(t) \leq 2t.
      \end{equation*} 
    \item If $\varphi \in \Delta_2 \cap \nabla_2$ we have the equivalences 
      \begin{equation*}\label{eq:varphi_equivalence}
        \varphi(t) \sim t \varphi'(t) \sim \varphi^*(\varphi'(t))  \sim \varphi^*(\varphi(t)/t).
      \end{equation*}
  \end{enumerate}
\end{lem}

We refer the reader to \cite{article:RaoRen91} for a proof; note that \ref{eq:young_inequality} and \ref{eq:varphi_inverse} are shown in Theorem I.3 and Proposition II.1(ii) respectively, and \ref{eq:varphi_equivalence} follows by convexity of $\varphi$ and \ref{eq:varphi_inverse}.

In particular \ref{eq:young_inequality} implies the equivalent definition
\begin{equation}\label{eq:conjugate_equiv}
  \varphi^*(s) = \sup_{t>0} \left(ts - \varphi(t)\right).
\end{equation} 
Also for any $\delta>0$, applying \ref{eq:young_inequality} with $s/\delta$ in place of $s$, we have
\begin{equation}\label{eq:delta_young}
  ts \leq \delta\varphi(t) + \delta \varphi^*(s/\delta).
\end{equation} 

If $\varphi$ is an $N$-function and $a > 0,$ following \cite{article:BalciEtAl20} we also introduce the \emph{shifted $N$-function}
\begin{equation}\label{eq:varphi_shifted}
  \varphi_a(t) = \int_0^t \frac{\tau \varphi'(\max\{a,\tau\})}{\max\{a,\tau\}} \,\d \tau.
\end{equation} 
Note that $\varphi_a(t) \sim t^2$ if $t \leq a$ and $\varphi_a(t) \sim \varphi(t)$ if $t\geq a$ (where the constant depends on $a$) and using (\ref{eq:conjugate_defn}) we have $(\varphi_a)^* = (\varphi^*)_{\varphi'(a)}.$ Further if $\varphi \in \Delta_2\cap\nabla_2,$ then the same holds for each $\varphi_a$ with $1 \leq \Delta_2(\varphi_a) \leq \Delta_2(\varphi),$ $1 \leq \nabla_2(\varphi_a) \leq \nabla_2(\varphi)$ for all $a>0.$ Also if $\varphi \in \Delta_2,$ then for each $M>0$ we have $\varphi_{1+a} \sim \varphi_{1+M}$ for all $0 \leq a \leq M$ (where the constant depends on $M, \Delta_2(\varphi)$).

\subsection{Orlicz-Sobolev spaces}

We will also define the natural function spaces associated to the functional (\ref{eq:functional_autonomous}), and establish some basic properties that we will use later.

\begin{defn}
  For an open set $\Omega \subset \bb R^n,$ a finite-dimensional normed space $(\bb V,\lvert\,\cdot\,\rvert)$ and $\varphi \in \Delta_2,$ we define the \emph{Orlicz space} $\LL^{\varphi}(\Omega,\bb V)$ as the space of $f \in \LL^1_{\loc}(\Omega,\bb V)$ for which
\begin{equation}
  \rho_{\Omega}^{\varphi}(\lvert f\rvert) := \int_{\Omega} \varphi(\lvert f \rvert) \,\d x < \infty,
\end{equation} 
which we equip with the \emph{Luxemburg norm}
\begin{equation}
  \norm{f}_{\LL^{\varphi}(\Omega,\bb V)} = \inf\left\{ \lambda>0 \colon \rho_{\Omega}^{\varphi}\left( \frac{\lvert f\rvert}{\lambda} \right) \leq 1 \right\}.
\end{equation} 
\end{defn}
Note that the $\Delta_2$-condition ensures that $\LL^{\varphi}(\Omega,\bb V)$ as defined is a linear space, and that $f \in \LL^1_{\loc}(\Omega,\bb V)$ lies in $\LL^{\varphi}(\Omega,\bb V)$ if and only if $\norm{f}_{\LL^{\varphi}(\Omega,\bb V)} < \infty.$ Also if $\Omega$ is bounded, we have $\LL^{\varphi}(\Omega,\bb V) = \LL^{\varphi_{a}}(\Omega,\bb V)$ for all $a>0,$ with $\varphi_a$ defined as in (\ref{eq:varphi_shifted}).

Given this, for $N, k, \ell \geq 1$ we can define the \emph{Orlicz-Sobolev spaces} $\WW^{k,\varphi}(\Omega,\bb M_{\ell})$ as the space of $k$-times weakly differentiable mappings $u \in \LL^{\varphi}(\Omega,\bb M_{\ell})$ such that $\nabla^j u \in \LL^{\varphi}(\Omega,\bb M_{\ell+j})$ for all $1 \leq j \leq k.$ We equip this space with the norm
\begin{equation}
  \norm{u}_{\WW^{k,\varphi}(\Omega,\bb M_{\ell})} = \sum_{j=0}^{k} \norm{\nabla^ju}_{\LL^{\varphi}(\Omega,\bb M_{\ell+j})}.
\end{equation} 

We also define $\WW^{k,\varphi}_0(\Omega,\bb M_{\ell})$ to be the closure of $\CC^{\infty}_c(\Omega,\bb M_{\ell})$ with respect to $\norm{\cdot}_{\WW^{k,\varphi}(\Omega,\bb M_{\ell})},$ and $\WW^{k,\varphi}_{\loc}(\Omega,\bb M_{\ell})$ to be the space of $u \in \WW^{k,1}_{\loc}(\Omega,\bb M_{\ell})$ such that the restriction $u\rvert_{\Omega'}$ lies in $\WW^{k,\varphi}(\Omega',\bb M_{\ell})$ for each compactly contained domain $\Omega' \Subset \Omega.$

These Orlicz-Sobolev spaces enjoy many of the familiar properties satisfied by the standard Sobolev spaces; see for instance \cite[Section 8]{book:AdamsFournier03}. We will record some specific results we need here.
This first result we need is the following Poincar\'e-Sobolev inequality, which is far from sharp, but will suffice for our purposes (for a sharp statement see \cite{article:Cianchi96}).

\begin{lem}\label{lem:poincare_sobolev}
  Let $\varphi \in \Delta_2$ and $\ell \geq 1$, then if $u \in \WW^{1,\varphi}(\Omega,\bb M_{\ell})$ we have $u \in \LL^{\varphi^p}_{\loc}(\Omega,\bb M_{\ell})$ for each $1 \leq p \leq \frac{n}{n-1}.$ Moreover for any $\B_R(x_0) \subset \Omega$ we have
  \begin{equation}
    \left( \dashint_{\B_R(x_0)} \varphi\left( \frac{\lvert u - (u)_{\B_R(x_0)}\rvert}{R} \right)^p \,\d x \right)^{\frac1p} \leq C\,\dashint_{\B_R(x_0)} \varphi(\lvert\nabla u\rvert)\,\d x,
  \end{equation} 
  where $C=C(n,\Delta_2(\varphi))>0.$ If $u \in \WW^{1,\varphi}_0(\B_R(x_0),\bb M_{\ell}),$ the same holds without subtracting the average.
\end{lem}

\begin{proof}
  We will establish the result for $R=1,$ from which the general case follows by rescaling. The result for $p=1$ follows by a standard application of the Riesz potential, as is shown in \cite{article:BhattacharyaLeonetti91} (see also \cite{article:DieningEttwein08}). For general $1 \leq p \leq \frac{n}{n-1}$ assume that $(u)_{\B_1(x_0)}=0$ by translation, and noting that $\lvert\nabla(\varphi(\lvert u \rvert))\rvert = \varphi'(\lvert u\rvert)\lvert\nabla u\rvert$ we can apply Young's inequality to bound
  \begin{equation}
    \begin{split}
      \dashint_{\B_1(x_0)} \left\lvert \nabla(\varphi(\lvert u\rvert))\right\rvert \,\d x 
      &\leq \dashint_{\B_1(x_0)} \varphi^*(\varphi'(\lvert u \rvert)) + \varphi(\lvert \nabla u\rvert) \,\d x \\
      &\leq \dashint_{\B_1(x_0)} \varphi(2\lvert u\rvert) + \varphi(\lvert \nabla u\rvert) \,\d x.
    \end{split}
  \end{equation} 
  From here we conclude that $\varphi(\lvert u\rvert) \in \WW^{1,1}(\B_1(x_0))$ using the $\Delta_2$-condition and the $p=1$ case, and so by the Gagliardo-Nirenberg inequality we have $\varphi(\lvert u\rvert) \in \LL^p(\B_1(x_0))$ for each $1 \leq p \leq \frac{n}{n-1}$ with the associated estimate
  \begin{equation}
    \left(\dashint_{\B_1(x_0)} \varphi(\lvert u\rvert)^p \,\d x \right)^{\frac1p} \leq C\,\dashint_{\B_1(x_0)} \lvert\nabla\varphi(\lvert u\rvert)\rvert \,\d x \leq C\,\dashint_{\B_1} \varphi(\lvert\nabla u\rvert) \,\d x,
  \end{equation} 
  as required. The $\WW^{1,\varphi}_0(\Omega;\mathbb M_{\ell})$ case can be found in \cite[Proposition 4.1]{article:ChlebickaEtAl19}.
\end{proof}

\begin{rem}\label{rem:higher_poincaresobolev}
  Applying this iteratively with $p=1$, we deduce for all $0 \leq j \leq k-1$ that
  \begin{equation}\label{eq:higher_poincaresobolev}
    \dashint_{\B_R(x_0)} \varphi\left(\frac{\lvert\nabla^j( u - a_{x_0,R})\rvert}{R^{k-j}}\right) \,\d x \leq C\,\dashint_{\B_R(x_0)} \varphi\left( \lvert\nabla^ku\rvert \right) \,\d x,
  \end{equation} 
  where $a_{x_0,R}$ is the unique polynomial of degree at most $k-1$ satisfying
  \begin{equation}
    \dashint_{\B_R(x_0)} \D^{\alpha}(u - a_{x_0,R}) \,\d x =0
  \end{equation} 
  for all $\lvert\alpha\rvert\leq k.$ 
  Similarly as in the $k=1$ case, we can omit the $a_{j,u}$ term if $u \in \WW^{k,\varphi}_0(\B_R(x_0),\bb M_{\ell})$.
\end{rem}

We will also need some results for affine functions in Section \ref{sec:local_min}, where we consider strong local minimisers. These can be deduced by combining the results in \cite[Lemmas 2.2, 2.3]{article:Habermann13} and \cite[Lemma 2]{article:Kronz02}.

\begin{lem}\label{lem:affine_approx}
  Let $\B_R(x_0)\subset \bb R^n$ be a ball, $N,\ell \geq 1,$ and $\varphi \in \Delta_2.$ Then if $u \in \WW^{1,\varphi}(\B_R(x_0),\bb M_{\ell})$, we have the affine function $A_{x_0,R} \colon \bb R^n \to \bb M_{\ell}$ defined to satisfy
  \begin{equation}
    A(x_0) =\dashint_{\B_R(x_0)} u\,\d x, \quad \nabla A = \frac{n+2}{R^2} \dashint_{\B_R(x_0)} u(x) \tensor (x-x_0)\,\d x
  \end{equation} 
  satisfies
  \begin{equation}\label{eq:affine_quasiminima}
    \dashint_{\B_R(x_0)} \varphi(\lvert u - A_{x_0,R}\rvert) \,\d x \leq C\, \dashint_{\B_R(x_0)} \varphi(\lvert u - A\rvert) \,\d x
  \end{equation} 
  for any other $A \colon \bb R^n \to \bb M_{\ell}$ affine, and we also have the estimates
  \begin{align}
    \varphi\left(\lvert \nabla A_{x_0,R} - (\nabla u)_{\B_R(x_0)}\rvert\right) &\leq C\,\dashint_{\B_R(x_0)} \varphi\left( \lvert \nabla u - (\nabla u)_{\B_R(x_0)}\rvert\right)  \,\d x.\label{eq:affine_derivative}\\
    \varphi\left(\lvert \nabla A_{x_0,R} - \nabla A_{x_0,\sigma R}\rvert\right) &\leq C\,\dashint_{\B_{\sigma R}(x_0)} \varphi\left( \frac{\lvert u - A_{x_0,R}\rvert}{\sigma R}\right)  \,\d x,\label{eq:affine_derivative2}
  \end{align}
  for all $\sigma \in (0,1).$
\end{lem}

Finally we record a interpolation estimate in the Orlicz scales, which is a straightforward adaption of the $\WW^{k,p}$ case (compare with the results in \cite[Section 5]{book:AdamsFournier03}).

\begin{lem}\label{lem:derivative_interpolation}
  Let $\varphi \in \Delta_2$ and $k \geq 0.$ Then there exists $C=C(n,k,\Delta_2(\varphi))>0$ such that for all $x_0 \in \bb R^n,$ $R>0$ and $u \in \WW^{k,\varphi}(\B_R(x_0),\bb M_{\ell}),$ we have the estimate
  \begin{equation}
    \int_{\B_R(x_0)} \varphi(\lvert\nabla^ju\rvert) \,\d x \leq C\int_{\B_R(x_0)} \varphi(\delta^{-j}\lvert u\rvert) + \varphi(\delta^{k-j} \lvert\nabla^ku\rvert) \,\d x
  \end{equation} 
  holds for all $\delta >0$.
\end{lem}

\begin{proof}[Sketch of proof]
  We start by observing the one-dimensional estimate
  \begin{equation}
    \varphi(\lvert f'(0)\rvert) \leq \frac{C}{\delta} \int_0^{\delta} \varphi(\delta^{-1}\lvert f(t)\rvert) + \varphi(\delta \lvert f''(t)\rvert) \,\d t,
  \end{equation} 
  valid for $C^2$ functions $f \colon [0,\delta_0) \to \mathbb R$ and $\delta \leq \delta_0$,
  following \cite[Lemma 5.4]{book:AdamsFournier03} using Jensen's inequality with $\varphi.$ 
  We moreover observe that this inequality remains valid for all $\delta>0$, where we understand the integral on the right-hand side to vanish on $[\delta_0,\infty)$.
  Then assuming $u$ is sufficiently regular, for $x \in \B_R(x_0)$ and $\omega \in S^{n-1}$ we can apply this to $f(t) = u(x+t\omega)$. Integrating the obtained bound over $x, \omega$ gives
  \begin{equation}
    \begin{split}
      \int_{\B_R(x_0)} \varphi(\lvert\nabla u\rvert) \,\d x 
      &\leq C \int_{\B_{R}(x_0)} \varphi\left(\delta^{-1}\lvert u(x)\rvert\right) + \varphi\left(\delta \lvert\nabla^2u(x)\rvert\right) \,\d x,
    \end{split}
  \end{equation} 
  nothing there is no contribution if $x + t\omega \notin \B_R(x_0)$, thereby
  establishing the $j=1,$ $k=2$ case. For the general case we proceed by induction; suppose the result holds for some $k \geq 2$ and $j=k-1,$ then for $\delta>0$ and $\mu \in (0,1)$ we can estimate
  \begin{equation}
    \begin{split}
      \int_{\B_R(x_0)} \varphi(\lvert\nabla^k u\rvert) \,\d x 
      &\quad\leq C \int_{\B_R(x_0)} \varphi(\delta^{-1}\lvert \nabla^{k-1} u\rvert) + \varphi(\delta \lvert \nabla^{k+1}u\rvert) \,\d x\\
      &\quad\leq C \int_{\B_R(x_0)} \varphi((\mu \delta)^{-1}\delta^{-1}\lvert u\rvert) + \varphi(\mu \lvert\nabla^k u\rvert) + \varphi(\delta \lvert \nabla^{k+1}u\rvert) \,\d x.
    \end{split}
  \end{equation} 
  By convexity of $\varphi$ we can choose $\mu>0$ sufficiently small so that $C\varphi(\mu\lvert\nabla^k u\rvert) \leq \frac12 \varphi(\lvert\nabla^ku\rvert)$, from which the estimate follows. A similar downward induction argument extends the result to all $1 \leq j \leq k-1.$
\end{proof}

\subsection{Linear elliptic estimates}

Our linearisation strategy will involve comparing our minimiser with solutions to a linearised system, for which we will need some solvability results. We will consider a bilinear form $\bb A$ on $\bb M_k$ satisfying the \emph{uniform Legendre-Hadamard ellipticity condition}
\begin{equation}\label{eq:general_LHelliptic}
  \lambda\lvert \xi\rvert^2 \lvert\eta\rvert^{2k} \leq \bb A[\xi\tensor\eta^k,\xi\tensor\eta^k] \leq \Lambda\lvert \xi\rvert^2 \lvert\eta\rvert^{2k}
\end{equation} 
for all $\xi \in \bb R^N,$ $\eta \in \bb R^n,$ where $0 < \lambda \leq \Lambda < \infty.$ Here we write $\eta^k = \eta \, \tensor \cdots \tensor\, \eta$ to denote the $k$-fold tensor product and identify elements $\xi \tensor \eta^{k} \in \bb M_k$ to send $(x_1,\dots,x_k) \to \xi\sum_{|\alpha|=k} x^{\alpha}\eta^{\alpha}.$ These generalise the rank-one matrices from the $k=1$ case. 

We will consider the operator
\begin{equation}
  \nabla^k : \bb A \nabla^ku = \sum_{\lvert\alpha\rvert=\lvert\beta\rvert=m} \D^{\beta}\left( \bb A_{\beta,\alpha} \D^{\alpha}u\right),
\end{equation} 
where the coefficients $\bb A_{\alpha\beta}$ of $\bb A$ satisfies $\xi_1 \cdot \bb A_{\alpha,\beta} \xi_2 = \bb A[\xi_1 \otimes \mathrm e^{\alpha},\xi_2 \otimes \mathrm e^{\beta}]$ for all $\xi_1, \xi_2 \in \mathbb R^N$.

\begin{rem}\label{rem:F_elliptic}
  For our harmonic approximation arguments, we will need the fact that the strict quasiconvexity condition \ref{eq:F_quasiconvex} implies that $\bb A[v,w] := F''(z_0)[v,w]$ is Legendre-Hadamard elliptic. To see this let $F$ satisfy Hypotheses \ref{hyp:generalgrowth_F}, and for $z_0 \in \bb M_k$ and $\xi \in \CC^{\infty}_c(\Omega,\bb R^N)$ consider the functional
  \begin{equation}
    \mathcal J(t) = \int_{\Omega} F(z_0+t\nabla^k\xi) - F(z_0) - \nu\, \varphi_{1+\lvert z_0\rvert}(t \nabla^k\xi) \,\d x.
  \end{equation} 
  By \ref{eq:F_quasiconvex} we have $\mathcal J(t) \geq 0$ for all $t$, so it is minimised at $t=0$. Therefore noting that $\varphi_{1+\lvert z_0\rvert}''(0)$ exists and by differentiating under the integral sign, we have $\mathcal J''(0)$ exists and is non-negative. That is, we have
  \begin{equation}
    \int_{\Omega} F''(z_0)\nabla^k\xi : \nabla^k \xi \,\d x \geq \nu \int_{\Omega} \frac{\varphi'(\max\{1,\lvert z_0\rvert\})}{\max\{1,\lvert z_0\rvert\}}  \lvert\nabla^k \xi\rvert^2 \,\d x.
  \end{equation} 
  From this we deduce that
  \begin{equation}\label{eq:F_elliptic}
    F''(z_0) (\xi \tensor \eta^k) : (\xi \tensor \eta^k) \geq \frac{\nu\,\varphi'(1)}{1+M} \lvert \xi\rvert^2 \lvert \eta\rvert^{2k}
  \end{equation} 
for all $z_0 \in \bb M_k$ with $\lvert z_0\rvert\leq M$ and $\xi \in \bb R^N, \eta \in \bb R^n.$ Note that since we normalised $\varphi(1)=1,$ the ellipticity constant only depends on $\nu, N, \Delta_2(\varphi).$ 
\end{rem}

\begin{prop}\label{prop:elliptic_solvability}
  Let $\varphi \in \Delta_2 \cap \nabla_2$ be an $N$-function and $\bb A$ be uniformly Legendre-Hadamard elliptic as above. Then if $\Omega \subset \bb R^n$ is a bounded smooth domain, the problem
  \begin{equation}\label{eq:linear_divform}
    \pdeproblem{(-1)^k \nabla^k : \bb A\nabla^k u}{(-1)^k \nabla^k :G}{\Omega,}{\partial_{\nu}^ju}{0}{\partial \Omega, \text{ for all }  1 \leq j \leq k-1,}
  \end{equation} 
  is uniquely solvable in $\WW_0^{k,\varphi}(\Omega,\bb R^N)$ for all $G \in \LL^{\varphi}(\Omega, \bb M_k),$ and the unique solution $u$ satisfies
  \begin{equation}\label{eq:elliptic_modular}
    \int_{\Omega} \varphi\left(\left\lvert \nabla^ku\right\rvert\right)\,\d x \leq C \int_{\Omega} \varphi(\lvert G\rvert) \,\d x,
  \end{equation} 
  where $C = C(n,N,k,\Omega,\lambda,\Lambda,\Delta_2(\varphi),\nabla_2(\varphi))>0$.
\end{prop}

We will apply this estimate on balls, on which we can note this estimate is scale invariant.
Note that the boundary condition (\ref{eq:linear_divform}) will be interpreted as simply requiring $u \in \WW^{k,\varphi}_0(\Omega,\bb R^N).$

This result is well-known and follows from the results in \textsc{Agmon, Douglis, \& Nirenberg} \cite{article:ADN2} in the $\LL^p$ setting, which extends to the Orlicz setting by interpolation. We will sketch a more elementary argument based on the interpolation results of \textsc{Stampacchia} \cite{article:Stampacchia65}, using a modern formulation using maximal functions (see \emph{e.g.}\,\cite{article:Krylov07,article:DongKim18}), which we combine with a Marcinkiewicz-type interpolation argument.
In the case $k=1$, these methods were also surveyed by the author in \cite[Chapter 3]{thesis:Irving22}.

\begin{proof}[Sketch of proof]
  We first establish the result for when $\varphi(t) = t^p,$ for which we let $L = (-1)^k \nabla^k : \bb A\nabla^k$ viewed as a linear operator $\WW^{k,p}_0(\Omega,\bb R^N) \to \WW^{-k,p}(\Omega,\bb R^N) \simeq \WW_0^{k,p'}(\Omega,\bb R^N)^*.$ When $p=2$ we can use the Plancherel theorem to show that for any $\omega \subset \bb R^n$ open we have
  \begin{equation}\label{eq:w_energy_estimate}
    \lambda \int_{\omega} \lvert \nabla^k u \rvert^2 \,\d x \leq \int_{\omega} \bb A[\nabla^k u, \nabla^k u] \,\d x
  \end{equation} 
  for all $u \in \WW^{k,2}_0(\omega,\bb R^N),$ from which we can deduce unique solvability in $\omega$ using the Lax-Milgram lemma. If $p\geq 2$ we will establish an a-priori estimate, so let $u \in \CC^{\infty}_c(\Omega,\bb R^N).$ Then if $x_0 \in\overline\Omega$ and $0<R<R_0$ with $R_0>0$ sufficiently small, we have the problem
  \begin{equation}
    \pdeproblem{Lw}{Lu}{\Omega \cap \B_R(x_0),}{\partial_{\nu}^ju}{0}{\partial \Omega \cap \B_R(x_0), \text{ for all } 1 \leq j \leq k-1}
  \end{equation} 
  admits a unique solution $w \in \WW^{k,2}_0(\Omega \cap \B_R(x_0),\bb R^N).$ Since the difference $v = u-w$ satisfies $Lv = 0$ in $\Omega \cap \B_R(x_0),$ by standard energy methods (see for instance \cite[Section 5.11]{book:Taylor1_11}) we have the uniform estimate
  \begin{equation}\label{eq:v_uniform_estimate}
    \sup_{\Omega \cap \B_{R/2}(x_0)} \lvert\nabla^{k+1}v\rvert \leq \frac{C}{R} \dashint_{\Omega \cap \B_R(x_0)} \lvert\nabla^ku\rvert^2 \,\d x.
  \end{equation} 
  By combining these, we will show for all $x_0 \in \overline\Omega$ and $\theta \in (0,1)$ that
  \begin{equation}\label{eq:pointwise_maximal}
    \mathcal M_{\Omega}^{\#}(\lvert\nabla^ku\rvert)(x_0) \leq C \left( \theta \mathcal M_{\Omega}\left(\lvert\nabla^ku\rvert^2\right)(x_0)^{\frac12}  + \theta^{-\frac n2}\mathcal M_{\Omega}\left(\lvert G\rvert^2\right)(x_0)^{\frac12}  \right),
  \end{equation} 
  where we define the localised maximal functions
  \begin{align}
    \label{eq:max_def}\mathcal M_{\Omega}(f)(x_0) &= \sup_{\substack{R>0}} \dashint_{\Omega \cap \B_R(x_0)} \lvert f\rvert \,\d x \\
    \mathcal M_{\Omega}^{\#}(f)(x_0) &= \sup_{\substack{R>0}} \dashint_{\Omega \cap \B_R(x_0)} \lvert f - (f)_{\Omega \cap \B_R(x_0)}\rvert \,\d x
  \end{align}
  for $f \in \LL^1_{\loc}(\Omega,\bb M_k).$ To see this, for $R<R_0$ note that the above estimates implies that
  \begin{equation}
    \begin{split}
      &\dashint_{\Omega \cap \B_{\theta R}(x)} \lvert \nabla^k u - (\nabla^k u)_{\Omega \cap \B_{\theta R}(x_0)}\rvert^2 \,\d x \\
      &\leq 2\dashint_{\Omega \cap \B_{\theta R}(x)} \lvert \nabla^k v - (\nabla^k v)_{\Omega \cap \B_{\theta R}(x_0)}\rvert^2 \,\d x
      + 2\dashint_{\Omega \cap \B_{\theta R}(x)} \lvert \nabla^k w - (\nabla^k w)_{\Omega \cap \B_{\theta R}(x_0)}\rvert^2 \,\d x \\
      &\leq C \theta^2 \dashint_{\Omega \cap \B_R(x_0)} \lvert\nabla^k u\rvert^2 \,\d x + C\theta^{-n} \dashint_{\Omega \cap \B_R(x_0)} \lvert G\rvert^2 \,\d x,
    \end{split}
  \end{equation} 
  using \eqref{eq:w_energy_estimate} for $w$ and \eqref{eq:v_uniform_estimate} for $v$.
  If $R>R_0$ similar estimates hold by a patching argument, from which the pointwise estimate (\ref{eq:pointwise_maximal}) follows. Now given $p \geq 2$, taking $\LL^{p/2}$ norms on both sides, we can use the Hardy-Littlewood maximal inequality and the Fefferman-Stein inequality in this setting (see \emph{e.g.}\,\cite{article:DRS10,article:DongKim18}) to deduce the estimate
  \begin{equation}
    \norm{\nabla^ku}_{\LL^p(\Omega,\bb M_k)} \leq C \left( \theta \norm{\nabla^ku}_{\LL^p(\Omega,\bb M_k)} + \theta^{-\frac n2} \norm{G}_{\LL^p(\Omega,\bb M_k)} \right),
  \end{equation}
  so if $u \in \CC^{\infty}_c(\Omega)$ we deduce $\LL^p$ estimates of the from \eqref{eq:elliptic_modular} for $p\geq 2.$ By density these estimates extend to all $u \in \WW^{k,p}_0(\Omega,\bb R^N),$ and by regularising $G \in \LL^p(\Omega,\bb M_k)$ and passing the limit using the a-priori estimate we can infer that $L$ is an isomorphism $\WW^{k,p}_0(\Omega,\bb R^N) \to \WW^{-k,p}(\Omega,\bb R^N)$ (injectivity follows form the $p=2$ case). Hence by duality (noting the adjoint operator $L^*$ is an isomorphism) unique solvability extends to the $1<p<2$ range.

  For the general Orlicz setting, we will extend the a-priori estimate \eqref{eq:elliptic_modular}, proven for $\varphi(t) = t^p$, to hold for a general $N$-function $\varphi \in \Delta_2 \cap \nabla_2$.
  This will involve a generalisation of the Marcinkiewicz interpolation theorem,  which is proven in \cite[Theorem 2]{article:Zygmund56} (see also \cite{article:Riordan55}); here one requires that there exists $1 < p_0 < p_1 < \infty$ and $C>0$ such that
  \begin{equation}
    \int_t^{\infty} s^{-p_1} \varphi(s) \frac{\d s}{s} \leq C t^{-p_1} \varphi(t), \quad \int_0^{t} s^{-p_0} \varphi(s) \frac{\d s}{s} \leq C t^{-p_0} \varphi(t)
  \end{equation} 
  for all $t \geq 0$.
  For general $\varphi \in \Delta_2\cap\nabla_2$, it follows from \cite[Theorems 11.7, 11.8]{book:Maligranda89} that such $p_0,p_1$ exist, thereby establishing \eqref{eq:elliptic_modular}.
  Now we can show that $L \colon \WW^{k,\varphi}_0(\Omega;\mathbb R^N) \to \WW^{-k,\varphi}_0(\Omega;\mathbb R^N)$ is an isomorphism by the same argument as in the $\LL^p$ case, noting that uniqueness follows from the inclusion $\WW^{k,\varphi}_0(\Omega) \subset \WW^{k,p_0}_0(\Omega)$.
\end{proof}

We will also need solvability results for when the right-hand side $g$ lies in $\WW^{k-1,\varphi}(\Omega,\bb M_{k-1}).$ This will follow from the above in conjunction with the following (non-optimal) representation theorem.

\begin{lem}\label{lem:negative_sobolev}
  Suppose $\varphi \in \Delta_2 \cap \nabla_2,$ $\ell \geq 0$ and $g \in \LL^{\varphi}(\B_R(x_0),\bb M_{\ell}),$ where $\B_R(x_0)\subset \bb R^n$ is any ball. Then there exists $G \in \LL^{\varphi^{\frac{n}{n-1}}}(\B_R(x_0),\bb M_{\ell+1})$ such that $-\nabla \cdot G = g$ in $\B_R(x_0),$ and we have the corresponding estimate
  \begin{equation}
    \left(\dashint_{\B_R(x_0)} \varphi(\lvert G\rvert)^{\frac{n}{n-1}} \,\d x\right)^{\frac{n-1}n} \leq C \dashint_{\B_R(x_0)} \varphi(R\lvert g\rvert) \,\d x.
  \end{equation} 
\end{lem}

\begin{proof}[Sketch of proof]
  We will consider the Newtonian potential (see \emph{e.g.}\,\cite[Section 7]{book:GilbargTrudinger98})
  \begin{equation}\label{eq:G_defn}
    G(x) = \frac{-1}{n\omega_n}\int_{\B_R(x_0)} \frac{g(y) \tensor (x-y)}{\lvert x- y\rvert^{n}} \,\d y,
  \end{equation} 
  which satisfies $-\nabla\cdot G = g$ in $\B_R(x_0)$ and we claim we have the singular integral estimates
  \begin{equation}
    \int_{\B_R(x_0)} \varphi(\lvert\nabla G\rvert) \,\d x \leq \int_{\B_R(x_0)} \varphi(\lvert g\rvert) \,\d x.
  \end{equation} 
  Indeed the case $\varphi(t) = t^p$ is proven in \cite[Theorem 9.9]{book:GilbargTrudinger98} (applied to $\mathbbm{1}_{\B_R(x_0)}g$), and we can extend to general $\varphi \in \Delta_2 \cap \nabla_2$ using the results of \cite{article:Zygmund56} analogously as in the above proof.
  We now apply Poincar\'e-Sobolev inequality (Lemma \ref{lem:poincare_sobolev}), which gives an extra term arising form the average; this can be bounded using the estimate
  \begin{equation}
    \begin{split}
      \int_{\B_R(x_0)} \varphi(\lvert G\rvert) \,\d x
      &\leq \frac{C}{R} \int_{\B_R(x_0)} \int_{\B_R(x_0)} \varphi(R\lvert g(y)\rvert) \lvert x- y\rvert^{1-n} \,\d x \,\d y \\
      &\leq C \int_{\B_R(x_0)} \varphi(\lvert R g(y)\rvert) \,\d y.
    \end{split}
  \end{equation} 
  using (\ref{eq:G_defn}), Jensen's inequality (applied to the measure $\d \mu = \frac CR\lvert x- y\rvert^{1-n} \,\d y$) and Fubini's theorem.
\end{proof}

\subsection{A Lipschitz truncation lemma}

We will establish the following higher-order version of the Lipschitz truncation lemma of \textsc{Acerbi \& Fusco} \cite{article:AcerbiFusco84,article:AcerbiFusco87}, which we will need for the harmonic approximation argument. The higher order case requires a more delicate extension argument, and was established in \textsc{Friesecke, James, \& M\"uller} \cite[Proposition A.2]{article:FrieseckeEtAl02} for $k=2$ using extension results in \textsc{Ziemer} \cite{book:Ziemer89}. These arguments straightforwardly generalise to the case of general $k$, which we record here.

\begin{prop}\label{prop:lipschitz_truncation}
  Let $\varphi \in \Delta_2 \cap \nabla_2$ and $q \in \WW^{k,\varphi}_0(\B_R(x_0),\bb R^N)$ with $\B_R(x_0) \subset \bb R^n$ a ball. Then for each $\lambda>0$ there exists $q_{\lambda} \in \WW^{k,\infty}_0(\B_R(x_0),\bb R^N)$ satisfying the following estimates.
  \begin{align}
    \norm{\nabla^kq_{\lambda}}_{\LL^{\infty}(\B_R(x_0),\bb M_k)} &\leq C_1\lambda, \\
    \int_{\B_R(x_0)} \varphi(\lvert\nabla^kq_{\lambda}\rvert) \,\d x & \leq C_2\int_{\B_R(x_0)} \varphi(\lvert\nabla^kq\rvert)\,\d x, \label{eq:lipschitz_truncation1}\\
    \varphi(\lambda) \mathcal L^n\left(\left\{ x \in \B_R(x_0) : q(x) \neq q_{\lambda}(x)\right\}\right) & \leq C_2\int_{\B_R(x_0)} \varphi(\lvert\nabla^kq\rvert)\,\d x,\label{eq:lipschitz_truncation2}
  \end{align}
  where $C_1 = C_1(n,N,k)$ and $C_2 = C_2(n,N,k,\Delta_2(\varphi),\nabla_2(\varphi)).$
\end{prop}

\begin{proof}
  Working componentwise we can assume $q$ is scalar-valued, and extending by zero we will also view $q \in \WW^{k,\varphi}(\mathbb R^n)$.
  and by translation and rescaling we will work with the unit ball $\B = \B_1(0)$ centred at the origin. 
  We will also work with the precise representatives of $q, \nabla q, \dots, \nabla^kq,$ and let $L_{\B}$ denote the set of Lebesgue points for every $\nabla^jq$, $0\leq j \leq k$.
  Then for $\lambda>0$ consider
  \begin{equation}
    H_{\lambda} = \left\{ x \in L_{\B} : \mathcal M(\lvert \D^{\beta} q\rvert) \leq \lambda \text{ for all } \lvert\beta \rvert\leq k \right\},
  \end{equation} 
  where $\mathcal M = \mathcal M_{\mathbb R^n}$ is the maximal operator defined as in \eqref{eq:max_def}.
  Note by boundedness of the maximal operator (using \cite[Theorem 1.2.1]{book:KokilashviliKrbec91} noting $\varphi \in \Delta_2 \cap \nabla_2$) we have
  \begin{equation}\label{eq:maximal_boundedness}
    \int_{\mathbb R^n} \varphi(\mathcal M(\lvert \nabla^jq\rvert)) \,\d x \leq C\int_{\B} \varphi(\lvert \nabla^jq\rvert) \,\d x
  \end{equation} 
  for each $0 \leq j \leq k$,
  from which it follows that $\mathcal L^n(\mathbb R^n \setminus H_{\lambda}) < \infty$.
  If $\mathcal L^n(\mathbb R^n \setminus H_{\lambda}) = 0$ then we can take $q_{\lambda} = q$,
  otherwise there exists a closed set $A_{\lambda} \subset H_{\lambda}$ for which $\mathcal L^n(\mathbb R^n \setminus A_{\lambda}) \leq 2 \mathcal L^n(\mathbb R^n \setminus H_{\lambda})$.
  We will now construct a $\WW^{k,\infty}$ map $q_{\lambda}$ satisfying $q=q_{\lambda}$ on $A_{\lambda}$, and show this satisfies all the claimed estimates.

  For this, we define for all $x \in A_{\lambda}$
  \begin{equation}
    P_x(y) = \sum_{\lvert\alpha\rvert \leq k-1} \frac{\D^{\alpha}q(x)}{\alpha!} (y-x)^{\alpha},
  \end{equation} 
  which by \cite[Theorem 3.4.1]{book:Ziemer89} satisfies
  \begin{equation}\label{eq:integral_qbound}
    \dashint_{\B_r(x)} \lvert \D^{\beta}q(y) - \D^{\beta}P_x(y)\rvert \,\d y \leq C\lambda r^{k-\lvert\beta\rvert}
  \end{equation} 
  for all $x \in A_{\lambda}$ and $\lvert \beta \rvert \leq k-1$.
  Hence by \cite[Theorem 3.5.7]{book:Ziemer89} we obtain the pointwise estimate
  \begin{equation}
    \lvert \D^{\beta}q(y) - \D^{\beta}P_x(y)\rvert \leq C \lambda \lvert x- y\rvert^{k-\lvert\beta\rvert}
  \end{equation} 
  for all $x,y \in A_{\lambda}$ and $\lvert\beta\rvert\leq k-1.$ For the boundary values we claim that
  \begin{equation}\label{eq:boundary_q_estimate}
    \lvert \D^{\beta}q(x)\rvert \leq C\lambda \dist(x,\partial \B)^{k-\lvert\beta\rvert}
  \end{equation} 
  holds for all $x \in A_{\lambda}$ and $\lvert\beta\rvert \leq k.$ To see this let $z \in \partial \B$ such that $\lvert x - z\rvert = d(x) = \dist(x,\partial \B),$ and note that $\B_{d(x)}(z) \subset \B_{2d(x)}(x).$ Since $\B$ is a ball it satisfies an external measure density condition of the from $\mathcal L^n(\B_r(z) \setminus \B) \geq \nu \mathcal L^n(\B_r(z))$ for all $z \in \partial \B$ and $0 < r< 1,$ so we have $\mathcal L^n(\B_{2d(x)}(x) \setminus  \B) \geq 2^{-n}\nu \mathcal L^n(\B_{2d(x)}(x)).$ Hence by a suitable Poincar\'e inequality (applied to the zero extension of $q$, using \cite[Corollary 4.5.2]{book:Ziemer89}) we have
  \begin{equation}
    \dashint_{\B_{2d(x)}(x)} \lvert \D^{\beta}q\rvert \,\d y \leq C \lambda \, d(x)^{k-\lvert\beta\rvert}.
  \end{equation} 
  We now combine this with \eqref{eq:integral_qbound} and use pointwise bounds for $P_x$, noting $x \in A_{\lambda}$, to estimate
  \begin{equation}
    \begin{split}
      \lvert \D^{\beta}q(x)\rvert &\leq \dashint_{\B_{2d}(x)} \lvert \D^{\beta}P_x(x) - \D^{\beta}P_x(y)\rvert \,\d y \\
                                 & \quad + \dashint_{\B_{2d}(x)} \lvert  \D^{\beta}P_x(y) - \D^{\beta}q(y)\rvert  + \lvert \D^{\beta}q(y)\rvert \,\d y 
                                 \leq C \lambda d(x)^{k-\lvert\beta\rvert},
    \end{split}
  \end{equation} 
  thereby establishing \eqref{eq:boundary_q_estimate}.
  Now we extend $P$ to $A_{\lambda} \cup \partial \B$ by setting $P_x = 0$ for all $x \in \partial \B,$ 
  so by the above estimates we have
  \begin{equation}\label{eq:qp_estimate}
    \lvert \D^{\beta}q(x) - \D^{\beta}P_x(y)\rvert \leq C_1 \lambda \lvert x - y\rvert^{k-\lvert\beta\rvert}
  \end{equation} 
  for all $x,y \in A_{\lambda} \cup \partial \B$.
  Hence letting $M > C_1$ to be determined, we define
  \begin{equation}
    \widetilde q_{\lambda}(y) = \sup_{x \in A_{\lambda} \cup \partial \B} \left(P_{x}(y) - M \lambda \lvert x- y\rvert^{k}\right),
  \end{equation} 
  which by choice of $M$ and \eqref{eq:qp_estimate} satisfies
  \begin{equation}
    \widetilde q_{\lambda}(x) = \twopartdef{q(x)}{\text{ if } x \in A_{\lambda},}{0}{\text{ if } x \in \partial \B.}
  \end{equation} 
  By increasing $M$ further if necessary, we claim this satisfies
  \begin{equation}\label{eq:firstextension_estimate}
    \lvert \widetilde q_{\lambda}(y) - P_x(y)\rvert \leq C\lambda \lvert x- y\rvert^k
  \end{equation} 
  for all $x \in A_{\lambda} \cup \partial \B$ and $y \in \overline \B.$ Indeed for $z \in A_{\lambda} \cup \partial \B$ we have
  \begin{equation}
    \begin{split}
      &(P_z(y) - M\lambda \lvert y - z\rvert^k) - P_x(y) \\
      &\quad \leq \lvert P_x(y) - P_z(y)\rvert - M\lambda \lvert y - z\rvert^k  \\
      &\quad \leq C\lambda \sum_{j=0}^{k-1} \lvert x - z\rvert^{k-j} \left(\lvert x - y\rvert^j  + \lvert z - y\rvert^j\right) - M \lambda \lvert y - z\rvert^k \\
      &\quad \leq C \lambda \lvert y - x\rvert^k + (C - M)\lambda \lvert y - z\rvert^k , 
    \end{split}
  \end{equation} 
  where we used Young's inequality and the fact that $\lvert x-z \rvert\leq \lvert x - y\rvert + \lvert y - z\rvert.$ Choosing $M \geq C$ and extremising over $z$ implies the estimate (\ref{eq:firstextension_estimate}), noting the lower bound follows from the definition.

  Now we can apply the extension result \cite[Theorem 3.6.2]{book:Ziemer89} to deduce the existence of $q_{\lambda} \in \WW^{1,\infty}_0( \B)$ such that $\norm{\nabla^kq}_{\LL^{\infty}( \B,\bb M_k)} \leq C\lambda,$ and such that $q = q_{\lambda}$ on $A_{\lambda}.$

  It remains to establish the estimates (\ref{eq:lipschitz_truncation1}), (\ref{eq:lipschitz_truncation2}). For (\ref{eq:lipschitz_truncation2}) we apply Markov's inequality to estimate
  \begin{equation}
    \begin{split}
      &\varphi(\lambda) \mathcal L^n(\{q \neq q_{\lambda}\}) 
        \leq 2\varphi(\lambda) \mathcal L^n(\mathbb R \setminus H_{\lambda}) \\
        &\quad\leq  C \sum_{j=0}^k\int_{\bb R^n} \varphi(\mathcal M(\lvert\nabla^jq\rvert)) \,\d x 
        \leq C \sum_{j=0}^k \int_{\B} \varphi(\lvert \nabla^j q\rvert) \,\d x
        \leq C \int_{\B} \varphi(\lvert\nabla^kq\rvert) \,\d x,
    \end{split}
  \end{equation} 
  where we have used \eqref{eq:maximal_boundedness} and the Poincar\'e inequality (Lemma \ref{lem:poincare_sobolev}) in the last two lines. For (\ref{eq:lipschitz_truncation1}) we can estimate
  \begin{equation}
    \int_{\B} \varphi(\lvert\nabla^kq_{\lambda}\rvert) \,\d x \leq \int_{\B} \varphi(\lvert\nabla^k q\rvert) \,\d x + \int_{\{q \neq q_{\lambda}\}} \varphi\left(\norm{\nabla^k q_{\lambda}}_{\LL^{\infty}(\B,\bb M_k)}\right) \,\d x,
  \end{equation} 
  from which we can conclude noting $\norm{\nabla^k q_{\lambda}}_{\LL^{\infty}(\B,\bb M_k)}\leq C\lambda$ and applying (\ref{eq:lipschitz_truncation2}).
\end{proof}

\section{Proof of the regularity theorem}

\subsection{Caccioppoli inequality of the second kind}

We will begin by establishing a Caccioppoli inequality of the second kind in this general growth setting, adapting the estimate of \textsc{Evans} \cite{article:Evans86}. In the case $k=1$, this argument in the Orlicz setting can be found in \cite{article:Dieningetal12}, and we will show this extends to the case of general $k.$

For this we fix $z_0 \in \bb M_k,$ and following \cite{article:AcerbiFusco87} we introduce the shifted integrand
\begin{equation}\label{eq:shifted_F}
  F_{z_0}(z) = F(z_0+z) - F(z_0) - F'(z_0)z.
\end{equation}
Then for each $M \geq 1,$ if $\lvert z_0\rvert  \leq M$ by distinguishing between the cases when $\lvert z_0 \rvert \leq M+1$ and $\lvert z_0 \rvert \geq M+1$ we have
\begin{align}
  \lvert F_{z_0}(z)\rvert  &\leq C\,\varphi_{1+\lvert z_0 \rvert}(\lvert z\rvert ),\label{eq:shifted_growth1} \\
  \lvert F_{z_0}'(z)\rvert  &\leq C\,\varphi_{1+\lvert z_0 \rvert}'(\lvert z\rvert ),\label{eq:shifted_growth2}
\end{align} 
where $C>0$ depends on $n, N,k,K, \Delta_2(\varphi)$ and $G(M) := \sup_{\lvert z\rvert  \leq 2M+1} \lvert F''(z)\rvert .$ Here the second estimate follows from the first by rank-one convexity of $F_{z_0}$ (implied by \ref{eq:F_quasiconvex}).

\begin{lem}[Caccioppoli inequality]\label{lem:caccioppoli}
  Let $u \in \WW^{k,\varphi}(\Omega,\bb R^N)$ be a minimiser of (\ref{eq:functional_autonomous}), where the integrand $F$ satisfies Hypotheses \ref{hyp:generalgrowth_F}, and let $M>0.$ Then if $a \colon \bb R^n \to \bb R^N$ is a $k$\th order polynomial such that $\lvert \nabla^k a\rvert \leq M,$ there for any $\B_R(x_0) \subset \Omega$ we have the estimate
  \begin{equation}
  \dashint_{\B_{R/2}(x_0)} \varphi_{1+M}(\lvert \nabla^k u - \nabla^k a\rvert) \,\d x \leq  C\,\dashint_{\B_{R}(x_0)} \varphi_{1+M}\left( \frac{\lvert u-a\rvert}{R^k} \right)  \,\d x.
  \end{equation} 
  where $C = C(n,N,k,K,\nu,\Delta_2(\varphi),M,G(M))>0$.
\end{lem}

\begin{proof}
  We will suppress the $x_0$ dependence to simplify notation. Fix $0<t<s<R,$ and let $\eta \in \CC^{\infty}_c(\B_R)$ be a radial cut-off such that $\mathbbm{1}_{\B_t} \leq \eta \leq \mathbbm{1}_{\B_s}$ and $\lvert \nabla^j\eta\rvert \leq \frac{C}{(s-t)^j}$ for each $1 \leq j \leq k.$ 
  Given $a$, we set $w = u-a,$ $\widetilde F = F_{\nabla^k a}$ and $\widetilde\varphi = \varphi_{1+\lvert\nabla^k a\rvert}$, noting that $\widetilde\varphi\sim\varphi_{1+M}$. 
  Observe that minimality of $u$ implies that   
  \begin{equation}
    \int_{\B_R} \widetilde F(\nabla^k w ) \,\d x \leq \int_{\B_R} \widetilde F(\nabla^kw + \nabla^k\xi) \,\d x
  \end{equation} 
  for any $\xi \in \WW^{k,\varphi}_0(\B_R)$, using that $\int_{\B_R} F'(\nabla^k a )\nabla^k\xi \,\d x = 0$.
  Then by the strict quasiconvexity condition \ref{eq:F_quasiconvex} we have
  \begin{equation}
    \begin{split}
      \nu\int_{\B_s} \widetilde\varphi(\lvert \nabla^k(\eta w)\rvert) \,\d x 
      &\leq \int_{\B_s} \widetilde F(\nabla^k(\eta w)) \,\d x \\
      &\leq \int_{\B_s} \widetilde F(\nabla^k w) \,\d x + \int_{\B_s} \widetilde F(\nabla^k(\eta w)) - \widetilde F(\nabla^kw) \,\d x.
   \end{split}
  \end{equation} 
  Now using the minimising property of $w$ and noting $\eta w = w$ on $\B_t$,
  \begin{equation}\label{eq:caccioppoli_middle1}
    \begin{split}
      &\int_{\B_t} \widetilde\varphi(\lvert \nabla^k(\eta w)\rvert) \,\d x \\
      &\quad\leq \frac1{\nu}\int_{\B_s} \widetilde F(\nabla^k((1-\eta)w)) \,\d x + \frac1{\nu}\int_{\B_s} \widetilde F(\nabla^k(\eta w)) - \widetilde F(\nabla^k w) \,\d x \\
      &\quad\leq C \int_{\B_s \setminus \B_t} \widetilde\varphi(\lvert \nabla^k((1-\eta)w)\rvert) + \widetilde\varphi(\lvert \nabla^k (\eta w) \rvert) + \widetilde\varphi(\lvert \nabla^k w\rvert) \,\d x \\
      &\quad\leq C \int_{\B_s \setminus \B_t} \widetilde\varphi(\lvert \nabla^k w\rvert) \,\d x + C \sum_{j=0}^{k-1}\int_{\B_s} \widetilde\varphi\left( \frac {\lvert\nabla^j w \rvert}{(s-t)^{k-j}} \right)\,\d x,
   \end{split}
  \end{equation} 
  using the $\Delta_2$-condition. By filling the hole, setting $\theta = \frac{C}{C+1} \in (0,1)$ gives
  \begin{equation}\label{eq:caccioppoli_middle2}
    \int_{\B_t} \widetilde\varphi(\lvert \nabla w\rvert) \,\d x \leq \theta \int_{\B_s} \widetilde\varphi(\lvert\nabla w\rvert) \,\d x + C_{\ast}\sum_{j=0}^{k-1}\int_{\B_s} \widetilde\varphi\left( \frac {\lvert\nabla^j w \rvert}{(s-t)^{k-j}} \right)\,\d x.
  \end{equation} 
  Now by the interpolation estimate (Lemma \ref{lem:derivative_interpolation}) we estimate
  \begin{equation}\label{eq:caccioppoli_middle3}
    \begin{split}
      C_{\ast}\sum_{j=0}^{k-1}\int_{\B_s} \widetilde\varphi\left( \frac {\lvert\nabla^j w \rvert}{(s-t)^{k-j}} \right)\,\d x 
      &\leq \frac{1-\theta}2 \int_{\B_s} \widetilde\varphi(\lvert\nabla^kw\rvert) \,\d x + C \int_{\B_s} \widetilde\varphi\left( \frac{\lvert w\rvert}{(s-t)^k} \right) \,\d x,
    \end{split}
  \end{equation} 
  allowing us to absorb the intermediate terms. Finally by an iteration argument (adapting \cite[Lemma 3.1]{article:Dieningetal12}) we deduce that
  \begin{equation}
    \int_{\B_{R/2}} \widetilde\varphi(\lvert\nabla^k w\rvert) \,\d x \leq C \int_{\B_R} \widetilde\varphi\left( \frac {\lvert w\rvert}{R^k} \right) \,\d x,
  \end{equation} 
  as required.
\end{proof}

\subsection{Harmonic approximation}

Our second ingredient will involve approximation by solutions $h$ to the linearised equation, adapting a recent strategy of \textsc{Gmeineder \& Kristensen} \cite{article:GmeinederKristensen19}, which has also been applied for instance in \cite{article:Gmeineder21,article:Franceschini19,article:Irving21,article:BarlinEtAl24}. A key feature of this argument is that we only use the Legendre-Hadamard ellipticity condition (\ref{eq:F_elliptic}) derived in Remark \ref{rem:F_elliptic} and that $u$ is $F$-extremal, so it can be applied more generally to infer regularity in situations where a suitable Caccioppoli inequality holds. This was exploited by the author in \cite{article:Irving21}, and we will also use this observation in Section \ref{sec:local_min} for the case of strong local minima.

We will also need some additional estimates on the integrand $F.$ Considering the shifted integrand $F_{z_0}$ defined in (\ref{eq:shifted_F}), where $z_0 \in \bb M_k$ with $\lvert z_0\rvert\leq M,$ we recall that $F_{z_0}''(0) = F''(z_0)$ satisfies the Legendre-Hadamard ellipticity bound $(\ref{eq:F_elliptic})$ by Remark \ref{rem:F_elliptic}. Further we also need the perturbation estimate
\begin{equation}\label{eq:F_pertubationestimate}
  \lvert F_{z_0}'(z) - F_{z_0}''(0)z\rvert  \leq C \omega_M(\lvert z\rvert )\left( \lvert z \rvert + \varphi_{\lvert z_0 \rvert}'(\lvert z\rvert)\right).
\end{equation} 
 Here $\omega_M$ denotes the modulus of continuity of $F''$ on $\{\lvert z\rvert  \leq 2M\},$ that is $\omega_M$ is a non-negative non-decreasing continuous concave function $[0,\infty) \to [0,1]$ satisfying $\omega_M(0)=0$ and
\begin{equation}
  \lvert F''(z) - F''(w)\rvert  \leq 2G(M) \omega_M(\lvert z-w\rvert )
\end{equation} 
for all $z,w \in \bb M_k$ such that $\lvert z\rvert ,\lvert w\rvert \leq 2M+1.$ The claimed estimate can be obtained by combining the above with the growth bound (\ref{eq:shifted_growth2}).

\begin{lem}[Harmonic approximation]\label{lem:harmonic_approximation}
  Let $u \in \WW^{k,\varphi}(\Omega,\bb R^N)$ be $F$-extremal, where $F$ satisfies Hypotheses \ref{hyp:generalgrowth_F}, and let $M>0,$ $\delta \in (0,1).$ Then if $a \colon \bb R^n \to \bb R^N$ is a $k$\th order polynomial such that $\lvert\nabla^k a\rvert\leq M,$ for any $\B_R(x_0) \subset \Omega$ the Dirichlet problem
  \begin{equation}\label{eq:harmonic_approximant}
    \pdeproblem{(-1)^k \nabla^k : F''(\nabla^k a)\nabla^k h}{0}{\B_R(x_0),}{\partial_{\nu}^jh}{\partial_{\nu}^j(u-a)}{\partial \B_R(x_0),\ \  1 \leq j \leq k-1,}
  \end{equation} 
  admits a unique solution $h \in (u-a)+\WW^{k,\varphi}_0(\Omega,\bb R^N),$ which satisfies the estimates
  \begin{equation}\label{eq:harmonic_estimate1}
    \dashint_{\B_R(x_0)} \varphi_{1+M}(\lvert\nabla^k h\rvert) \,\d x  \leq C\,\dashint_{\B_R(x_0)} \varphi_{1+M}(\lvert\nabla^k(u-a)\rvert) \,\d x,
  \end{equation} 
  and
  \begin{equation}\label{eq:harmonic_approximation}
    \begin{split}
      &\dashint_{\B_R(x_0)} \varphi_{1+M}\left( \frac{\lvert \nabla^{k-1}(u-a-h)\rvert }{R} \right) \,\d x \leq \delta\,\dashint_{\B_R(x_0)} \varphi_{1+M}\left(\nabla^k(u-a)\right) \,\d x \\
      &\qquad +C_{\delta}\,\gamma_M\left( \dashint_{\B_R} \varphi_{1+M}\left( \lvert\nabla^k(u-a)\rvert \right)  \,\d x \right)\dashint_{\B_R} \varphi_{1+M}\left( \lvert\nabla^k(u-a)\rvert \right)  \,\d x.
    \end{split}
  \end{equation} 
  Here $\gamma_M : [0,\infty) \to [0,\infty)$ is a non-decreasing continuous function such that $\gamma_M(0)=0$ and we have $C_{\delta} = C(n,N,k,K,\nu,\Delta_2(\varphi),\nabla_2(\varphi),M,G(M),\delta)>0.$
\end{lem}

\begin{proof}
  Suppressing the $x_0$-dependence, set $w=u-a,$ $\widetilde F = F_{\nabla^k a},$ and $\widetilde\varphi = \varphi_{1+\lvert\nabla^k a\rvert}$ noting that $\widetilde\varphi \sim \varphi_{1+M}.$ We know by Proposition \ref{prop:elliptic_solvability} that a unique solution $h$ exists and satisfies the modular estimate (\ref{eq:harmonic_estimate1}), and since $w$ is $\widetilde F$-extremal, using the respective weak formulations (\ref{eq:F_extremal}), (\ref{eq:harmonic_approximant}) we have
  \begin{equation}\label{eq:linearisation_step1}
    \int_{\B_R} \widetilde F''(0)\nabla^k(w-h) : \nabla^k q \,\d x = \int_{\B_R} (\widetilde F''(0)\nabla^k w - \widetilde F'(\nabla^k w)) : \nabla^k q \,\d x,
  \end{equation} 
  for $q \in \WW^{k,\infty}_0(\B_R,\bb R^N).$ Following \cite{article:GmeinederKristensen19} we wish to choose our test function so we obtain $\widetilde\varphi(R^{-1}\lvert\nabla^{k-1}( w- h)\rvert)$ on the left-hand side, so to achieve this we choose $q$ to be the solution of the dual problem
  \begin{equation}\label{eq:dualproblem}
    \pdeproblem{(-1)^k \nabla^k: F''(\nabla^k a)\nabla^k q}{(-1)^{k-1}\nabla^{k-1}: g}{\B_R,}{\partial_{\nu}^jq}{0}{\partial \B_R, \ \ 0 \leq j \leq k-1,}
  \end{equation} 
  where
  \begin{equation}
    g = \widetilde\varphi\left(\frac{\lvert \nabla^{k-1}(w- h)\rvert}{R} \right)\frac{\nabla^{k-1}(w-h)}{\lvert \nabla^{k-1}(w-h)\rvert^2}.
  \end{equation}
  Since $t \mapsto \frac{\widetilde\varphi(t)}t$ is an $N$-function satisfying $\widetilde\varphi^*(\widetilde\varphi(t)/t) \sim \widetilde\varphi(t)$ uniformly in $t,$ we have $g$ is well-defined and satisfies the estimate
  \begin{equation}\label{eq:g_estimates}
    \dashint_{\B_R} \widetilde\varphi^*(R\lvert g\rvert)\,\d x \leq C\,\dashint_{\B_R} \widetilde\varphi\left(\frac{\lvert \nabla^{k-1}(w-h)\rvert}{R} \right) \,\d x.
  \end{equation}  
  Hence by Lemma \ref{lem:negative_sobolev} we can write $g = - \nabla \cdot G$ with $G \in \LL^{\varphi^{\frac{n}{n-1}}}(\B_R,\bb M_k),$ so then Proposition \ref{prop:elliptic_solvability} gives the existence of a unique $q \in \WW^{k,\varphi^{\frac{n}{n-1}}}_0(\B_R,\bb R^N)$ and a corresponding modular estimate, which we combine with (\ref{eq:g_estimates}) to obtain 
  \begin{equation}\label{eq:q_mainestimates}
    \left( \dashint_{\B_R} \widetilde\varphi^*\left( \lvert\nabla^kq\rvert\right)^{\frac{n}{n-1}} \,\d x \right)^{\frac{n-1}n} \leq C \,\dashint_{\B_R} \widetilde\varphi\left( \frac{\lvert \nabla^{k-1}(w-h)\rvert}{R} \right) \,\d x.
  \end{equation} 
  In general we cannot use $q$ in \eqref{eq:linearisation_step1} since it may not lie in $\WW^{k,\infty}(\B_R,\mathbb R^N)$, so we will need to take a higher order Lipschitz truncation. Letting $\lambda>0$ to be determined, applying Proposition \ref{prop:lipschitz_truncation} with the $N$-function $(\widetilde\varphi^*)^{\frac{n}{n-1}}$ gives $q_{\lambda} \in \WW^{k,\infty}(\B_R,\bb R^N)$ satisfying $\norm{\nabla^kq_{\lambda}}_{\LL^{\infty}(\B_R,\bb M_k)} \leq C\lambda$ and
  \begin{align}
    \dashint_{\B_R} \widetilde\varphi^*\left(\lvert\nabla^k q_{\lambda}\rvert\right)^{\frac{n}{n-1}} \,\d x &\leq C\,\dashint_{\B_R} \widetilde\varphi^*\left(\lvert\nabla^k q\rvert\right)^{\frac{n}{n-1}} \,\d x,\label{eq:qlambda_psibound} \\ 
    \widetilde\varphi^*(\lambda)^{\frac{n}{n-1}}\mathcal L^n\left(\left\{ x \in \B_R : q(x) \neq q_{\lambda}(x) \right\}\right)^{\frac{n}{n-1}} &\leq C \int_{\B_R} \widetilde\varphi^*\left(\lvert \nabla^k q\rvert\right)^{\frac{n}{n-1}} \,\d x. \label{eq:qlambda_measure}
  \end{align} 
  For this choice of $q_{\lambda}$ we have for each $\delta \in (0,1),$
  \begin{equation*}
    \begin{split}
      &\dashint_{\B_R} \widetilde\varphi\left( \frac{\lvert\nabla^{k-1}(w-h)\rvert}{R} \right) \,\d x \\
      &= \dashint_{\B_R} (\widetilde F''(0)\nabla^k w - \widetilde F'(\nabla^k w)) : \nabla^k q_{\lambda} \,\d x 
      +\dashint_{\B_R} \widetilde F''(0)\nabla^k w  : (\nabla^k q - \nabla^k q_{\lambda}) \,\d x \\
      &\leq C\lambda \,\dashint_{\B_R} \omega_M(\lvert\nabla^kw\rvert) \left( \lvert\nabla^kw\rvert+\widetilde\varphi'(\lvert\nabla^kw\rvert) \right) \,\d x 
      + C\,\dashint_{\B_R} \lvert\nabla^kw\rvert \lvert \nabla^kq - \nabla^kq_{\lambda}\rvert \,\d x\\
      &\leq C \,\dashint_{\B_R} 2 \delta\,\widetilde\varphi(\lvert\nabla^k w\rvert) + \delta\,\widetilde\varphi^*(\widetilde\varphi'(\lvert\nabla^k w\rvert)) \,\d x \\
      &\quad+C \,\dashint_{\B_R} \delta\,\widetilde\varphi\left( \frac{\lambda}{\delta}\, \omega_M(\lvert\nabla^k w\rvert) \right) + \delta\,\widetilde\varphi^*\left( \frac{\lambda}{\delta}\,\omega_M(\lvert\nabla^k w\rvert)\right)\,\d x 
      + C \,\dashint_{\B_R} \delta\,\widetilde\varphi^*\left( \frac{\lvert\nabla^kq-\nabla^kq_{\lambda}\rvert}{\delta} \right) \,\d x \\
      &= C \left( \mathrm I + \mathrm{I\!I} + \mathrm{I\!I\!I} \right),
    \end{split}
  \end{equation*} 
  where we have used the perturbation estimate (\ref{eq:F_pertubationestimate}) and Young's inequality (\ref{eq:delta_young}). It remains to estimate each of these terms separately; for the first inequality we use the fact that $\widetilde\varphi^*(\widetilde\varphi'(t)) \leq C\varphi(t)$ to estimate
  \begin{equation}
    \mathrm I \leq C \delta \dashint_{\B_R} \widetilde\varphi(\lvert\nabla^kw\rvert) \,\d x,
  \end{equation} 
  and for the second term we use the fact that $\omega_M \leq 1$ and apply Jensen's inequality to the concave function $\omega_M \circ \widetilde\varphi^{-1}$ to estimate
  \begin{equation}
    \begin{split}
      \mathrm{I\!I} &\leq \delta \left( \widetilde\varphi\left( \frac{\lambda}{\delta} \right) +\widetilde\varphi^*\left( \frac{\lambda}{\delta} \right) \right) \dashint_{\B_R} \omega_M(\lvert\nabla^kw\rvert)\,\d x\\
           &\leq \delta^{1-\max\{p_{\varphi},p_{\varphi^*}\}} \left( \widetilde\varphi\left(\lambda \right) +\widetilde\varphi^*\left(\lambda \right) \right) \omega_M\circ \widetilde\varphi^{-1}\left( \dashint_{\B_R} \widetilde\varphi(\lvert\nabla^kw\rvert)\,\d x\right).
    \end{split}
  \end{equation} 
  For the third term we apply H\"older's inequality along with the Lipschitz truncation estimates (\ref{eq:qlambda_psibound}), (\ref{eq:qlambda_measure}) to estimate
  \begin{equation}
    \begin{split}
      \mathrm{I\!I\!I} &\leq \delta^{1-p_{\varphi^*}} \left( \frac{\mathcal L^n(\B_R \cap \{ q \neq q_{\lambda}\})}{\mathcal L^n(\B_R)} \right)^{\frac1n} \left( \dashint_{\B_R} \widetilde\varphi^*(\lvert\nabla^k(q -q_{\lambda})\rvert)^{\frac{n}{n-1}} \,\d x \right)^{\frac{n-1}n} \\
              &\leq C_{\delta}\left( \frac{\dashint_{\B_R} \widetilde\varphi^*(\lvert\nabla^k q\rvert)^{\frac{n}{n-1}}\,\d x}{\widetilde\varphi^*(\lambda)^{\frac{n}{n-1}}} \right)^{\frac1{n}} \left( \dashint_{\B_R} \widetilde\varphi^*(\lvert\nabla^k q\rvert)^{\frac{n}{n-1}} \right)^{\frac{n-1}n} \\
              &\leq C_{\delta}\left( \frac{\dashint_{\B_R} \widetilde\varphi\left(\frac{\lvert\nabla^{k-1}(w-h)\rvert}{R}\right)\,\d x}{\widetilde\varphi^*(\lambda)} \right)^{\frac1{n-1}} \dashint_{\B_R} \widetilde\varphi\left(\frac{\lvert\nabla^{k-1}(w-h)\rvert}{R}\right)\,\d x.
    \end{split}
  \end{equation} 
  We can now choose $\lambda>0$ sufficiently large so the third term can be absorbed to the left-hand side; for this we can take $\lambda$ to satisfy
  \begin{equation}
    C_{\delta} \left(\frac{\dashint_{\B_R} \widetilde\varphi\left(\frac{\lvert\nabla^{k-1}(w-h)\rvert}{R}\right) \,\d x}{\widetilde\varphi^*(\lambda)}\right)^{\frac1{n-1}}  = \frac12.
  \end{equation} 
  Now using the doubling property for $\widetilde\varphi, \widetilde\varphi^*$ and writing $\vartheta_M(t) = t + \widetilde\varphi \circ (\widetilde\varphi^*)^{-1}(t)$ we get
  \begin{equation}
    \begin{split}
      &\dashint_{\B_R} \widetilde\varphi\left( \frac{\lvert \nabla^{k-1}(w- h)\rvert}{R} \right) \,\d x 
      \leq C \delta \dashint_{\B_R} \widetilde\varphi(\lvert\nabla^kw\rvert) \,\d x \\
      &\quad + C_{\delta} \vartheta_M\left(\dashint_{\B_R} \widetilde\varphi\left( \frac{\lvert\nabla^{k-1}(w-h)\rvert}{R} \right) \,\d x\right) \omega_M\circ\widetilde\varphi^{-1}\left( \dashint_{\B_R} \widetilde\varphi(\lvert \nabla w\rvert)\,\d x \right).
    \end{split}
  \end{equation} 
  Finally note that since $w-h \in \WW^{k,\varphi}_0(\B_R,\bb R^N),$ by the Poincar\'e inequality and (\ref{eq:harmonic_estimate1}) we have
  \begin{equation}
    \begin{split}
      \dashint_{\B_R(x_0)} \widetilde\varphi\left(\frac{\lvert \nabla^{k-1}(w- h)\rvert}{R} \right) \,\d x
      &\leq C\,\dashint_{\B_R(x_0)} \widetilde\varphi\left(\lvert \nabla^k(w-h)\rvert\right)\,\d x \\
      &\leq C\,\dashint_{\B_R(x_0)} \widetilde\varphi\left( \lvert\nabla^kw\rvert  \right)\,\d x,
  \end{split}
  \end{equation} 
  so setting $\gamma_M(t) = \frac{\vartheta_M(t)}t\, \omega_M \circ \widetilde\varphi^{-1}(t)$ the result follows.
\end{proof}

\subsection{Excess decay estimate}

We now combine the above estimates to conclude, which will involve establishing estimates for the \emph{excess energies} defined for $M>0$ as
\begin{equation}
  \mathrm E_M(x,r) = \dashint_{\B_r(x)} \varphi_{1+M}\left(\left\lvert\nabla^k u - (\nabla^k u)_{\B_r(x)}\right\rvert\right) \,\d y. \label{eq:excess_energy_M}
\end{equation} 

\begin{lem}[Excess decay estimate]\label{lem:excess_decay}
  Let $u \in \WW^{k,\varphi}(\Omega,\bb R^N)$ be a minimiser of (\ref{eq:functional_autonomous}) where the integrand $F$ satisfies Hypotheses \ref{hyp:generalgrowth_F}, and let $M>0,$ $\delta>0.$ Then for all $\sigma \in \left(0,\frac14\right)$ and $\delta \in (0,1),$ letting $\gamma_M$ as in Lemma \ref{lem:harmonic_approximation}, if $\B_R(x_0)\subset \Omega$ such that $\lvert(\nabla^ku)_{\B_R(x_0})\rvert \leq M$ and $\mathrm E_M(x_0,R)\leq 1$ we have
  \begin{equation}
    \mathrm E_M(x_0,\sigma R) \leq C \left( \sigma + C_{\sigma}\delta + C_{\sigma,\delta}\gamma_M(\mathrm E_M(x_0,R)) \right) \mathrm E_M(x_0,R),
  \end{equation} 
  where the constants depend also on $n, N, k,K,\nu \Delta_2(\varphi),\nabla_2(\varphi), M$ and $G(M).$
\end{lem}

\begin{proof}
  We let $a_1 : \bb R^n \to \bb R^N$ be a $k$\th order polynomial satisfying
  \begin{equation}
    \dashint_{\B_R(x_0)} \D^{\alpha}(u-a_1)\,\d x = 0
  \end{equation} 
  for all $\lvert\alpha\rvert\leq k$ and apply Lemma \ref{lem:harmonic_approximation} on $\B_R(x_0)$ with $a_1,$ obtaining the unique solution $h$ to
  \begin{equation}
    \pdeproblem{(-1)^k \nabla^k : F''(\nabla^k a_1)\nabla^kh}{0}{\B_R}{\partial_{\nu}^jh}{\partial_{\nu}^j(u-a_1)}{\partial \B_R,\ \  0 \leq j \leq k-1.}
  \end{equation} 
  Combining the Poincar\'e inequality (Remark \ref{rem:higher_poincaresobolev}) with the approximation estimate (\ref{eq:harmonic_approximation}) from Lemma \ref{lem:harmonic_approximation} we have
  \begin{equation}\label{eq:recast_harmonicexcess}
    \begin{split}
      &\int_{\B_R(x_0)} \varphi_{1+M}\left( \frac{\lvert u - a_1-h\rvert}{R^k} \right) \,\d x \\
      &\quad \leq\int_{\B_R(x_0)} \varphi_{1+M}\left( \frac{\lvert \nabla^{k-1}(u - a_1-h)\rvert}{R} \right) \,\d x \\
      &\quad\leq \left(\delta + C_{\delta}\gamma_M\left(\mathrm E_M(x_0,R) \right)\right)\mathrm E_M(x_0,R).
    \end{split}
  \end{equation}
  Now let $a_2(x) = \sum_{\lvert\alpha\rvert\leq k} \frac{\D^{\alpha}h(x_0)}{\alpha!} (x-x_0)^{\alpha},$ and put $a=a_1+a_2.$ This satisfies $\varphi_{1+M}(\lvert\nabla^ka_2\rvert) \leq C(M+1),$ so $\lvert\nabla^ka\rvert \leq C(M)$ and hence $\varphi_{1+\lvert\nabla^ka\rvert} \sim \varphi_{1+M}.$ Therefore applying the Caccioppoli inequality (Lemma \ref{lem:caccioppoli}) with $a$ on $\B_{2\sigma R}(x_0)$ we obtain
  \begin{equation}
    \begin{split}
      \mathrm E_M(x_0,\sigma R) &\leq C\,\dashint_{\B_{2\sigma R}(x_0)} \varphi_{1+M}\left( \frac{\lvert u - a\rvert}{(2\sigma R)^k} \right) \,\d x\\
                      &\leq C_{\sigma} \dashint_{\B_R(x_0)} \varphi_{1+M}\left( \frac{\lvert u-a_1-h\rvert}{R^k} \right) \,\d x 
                      + C\,\dashint_{\B_{2\sigma R}(x_0)} \varphi_{1+M}\left( \frac{\lvert h-a_2\rvert}{(2\sigma R)^k} \right) \,\d x.
    \end{split}
  \end{equation} 
  Now as $h$ is $F''(\nabla^ka)$-harmonic, by interior regularity estimates (see for instance \cite[Section 5.11]{book:Taylor1_11}) and \eqref{eq:harmonic_estimate1} we have
  \begin{equation}
    \sup_{\B_{2\sigma R}} \frac{\lvert h(x) - a_2(x)\rvert}{(2\sigma R)^k} \leq \frac{2\sigma R}{(k+1)!} \sup_{\B_{R/2}(x_0)} \lvert\nabla^{k+1}h\rvert \leq C\varphi_{1+M}^{-1}\left(\sigma \mathrm E_M(x,R)\right),
  \end{equation} 
  so together with (\ref{eq:recast_harmonicexcess}) we obtain
  \begin{equation}
    \begin{split}
      \mathrm E_M(x_0,\sigma R) \leq C_{\sigma} \left( \delta + C_{\delta} \gamma_M(\mathrm E_M(x_0,R)) \right) \mathrm E_M(x_0,R) + C \sigma \mathrm E_M(x_0,R),
    \end{split}
  \end{equation} 
  which is the desired estimate.
\end{proof}

We can now conclude in the usual way.

\begin{proof}[Proof of Theorem {\ref{thm:main_epsreg}}]
  Suppose $\B_R(x_0) \subset \Omega$ such that $\lvert (\nabla^ku)_{\B_R(x_0)}\rvert\leq M,$ and let $x \in \B_{R/2}(x_0).$ Then since $\lvert(\nabla^ku)_{\B_{R/2}(x)}\rvert\leq 2^nM$ and $\mathrm E_M(x,R/2) \leq 2^n\eps \leq 1$ shrinking $\eps$ if necessary, we can apply Lemma \ref{lem:excess_decay} above with $M$ replaced by $2^nM$ obtain
  \begin{equation}
    \mathrm E_M(x,\sigma R/2) \leq C\left(\sigma + C_{\sigma}\delta + C_{\sigma,\delta} \gamma_M(\eps) \right)\mathrm E_M(x,R/2),
  \end{equation} 
  for $\sigma \in \left( 0,\frac14 \right)$ and $\delta>0.$ We now choose $\sigma>0$ sufficiently small to ensure $C\sigma < \frac13\sigma^{2\alpha}$ and $\delta>0$ small enough so $CC_{\sigma}\delta < \frac13\sigma^{2\alpha},$ and finally $\eps>0$ so that $CC_{\sigma,\delta}\gamma_M(\eps) < \frac13\sigma^{2\alpha},$ which gives
  \begin{equation}
    \mathrm E_M(x,\sigma R/2) \leq \sigma^{2\alpha}\mathrm E_M(x,R/2).
  \end{equation} 
  We claim, by iterating the above, that we have
  \begin{align}
    \lvert (\nabla^k u)_{\B_{\sigma^{j-1}R/2}(x)}\rvert &\leq 2^{n+1}M, \label{eq:iterated_averages}\\
    \mathrm E_M(x,\sigma^jR/2) &\leq \sigma^{2j\alpha} \mathrm E_M(x,R/2) \label{eq:iterated_decay}
  \end{align}
  for all $j \geq 1,$ provided $\eps>0$ is sufficiently small. 
  Indeed proceeding inductively, if this holds for for all  $1 \leq i \leq j,$ then note first that for such $i$ we use Jensen's inequality to bound
  \begin{equation}\label{eq:average_differences}
    \begin{split}
      &\lvert (\nabla^k u)_{\B_{\sigma^iR/2}(x)} - (\nabla^ku)_{\B_{\sigma^{i-1}R/2}(x)}\rvert \\
      &\quad\leq \sigma^{-n} \dashint_{\B_{\sigma^{i-1}R/2}(x)} \lvert\nabla^ku - (\nabla^ku)_{\B_{\sigma^{i-1}R/2}(x)}\rvert \,\d x \\
      &\quad\leq \sigma^{-n} \varphi_{1+M}^{-1}\left(  \mathrm E_M(x,\sigma^{i-1}R/2) \right) 
      \leq \sigma^{-n}\varphi_{1+M}^{-1}\left(\sigma^{2(i-1)\alpha}\eps\right).
    \end{split}
  \end{equation} 
  Hence by shrinking $\eps>0$ further if necessary to ensure that $\varphi_{1+M}^{-1}(\delta\eps) \leq C\sqrt{\delta}$ for all $\delta \in (0,1),$ we can estimate
  \begin{equation}
    \begin{split}
      \lvert (\nabla^k u)_{\B_{\sigma^{j}R/2}(x)}\rvert 
      &\leq 2^nM + \sum_{i=1}^{j} \lvert (\nabla^k u)_{\B_{\sigma^iR/2}(x)} - (\nabla^ku)_{\B_{\sigma^{i-1}}(x)}\rvert \\
      &\leq 2^nM + \sum_{i=0}^{j-1} \varphi_{1+M}^{-1}(\sigma^{2i\alpha}\eps)
      \leq 2^nM + \eps\sum_{i=0}^{\infty} \sigma^{i\alpha},
    \end{split}
  \end{equation} 
  which can be chosen to be less than $2^{n+1}M$ if $\eps>0$ is sufficiently small, establishing (\ref{eq:iterated_averages}). Therefore we can iteratively apply the excess decay estimate with the same parameters to deduce (\ref{eq:iterated_decay}). Hence for any $r \in (0,R/2)$ we have $E(x,r) \leq \varphi_{1+M}(2^{n+1}\eps) \leq 1$ and so
  \begin{equation}
    \dashint_{\B_r(x)} \lvert \nabla^k u - (\nabla^k u)_{\B_r(x)}\rvert \,\d y \leq C\sqrt{E(x,r)} \leq Cr^{\alpha}.
  \end{equation} 
  Since this holds for all $x \in \B_{R/2}(x)$ and $0 < r < R/2,$ by the Campanato-Meyers characterisation of H\"older continuity (see for instance \cite[Theorem 2.9]{book:Giusti03}) it follows that $\nabla^k u$ is $\CC^{0,\alpha}$ in $\B_{R/2}(x_0).$
\end{proof}

\section{Extension to strong local minimisers}\label{sec:local_min}

We will conclude our discussion with a straightforward extension of our main theorem to the setting of $\WW^{k,\psi}$-local minimisers. To be more precise, we consider the following.

\begin{defn}
  Let $F$ satisfy Hypotheses \ref{hyp:generalgrowth_F}, and $\psi$ be an $N$-function. Then we say $u \in \WW^{k,\varphi}(\Omega,\bb R^N)$ is a \emph{strong $\WW^{k,\psi}$-local minimiser} if there exists $\delta>0$ such that if $\xi \in \WW^{k,\psi}_0(\Omega,\bb R^N)$  with $\int_{\Omega} \psi(\lvert\nabla^k \xi\rvert) \,\d x < \delta,$ then we have 
  \begin{equation}\label{eq:minimising_property}
    \mathcal F(u) \leq \mathcal F(u+\xi).
  \end{equation} 
  We also say $u$ is a \emph{$\WW^{k,\infty}$-local minimiser} if (\ref{eq:minimising_property}) is satisfied whenever $\xi \in \WW^{k,\infty}_0(\Omega,\bb R^N)$ such that $\norm{\nabla^k\xi}_{\LL^{\infty}(\Omega,\bb M_k)} < \delta.$
\end{defn}

The regularity of strong $\WW^{1,q}$-minimisers was first considered by \textsc{Kristensen \& Taheri} \cite{article:KristensenTaheri03}, who established a partial regularity theorem in the case of superquadratic ($p\geq 2$) growth. These results have been extended for instance in \cite{article:CarozzaNapoli03,article:SchemmSchmidt09,article:CamposCordero17,article:CamposCordero21}.

\begin{thm}[$\eps$-regularity theorem for local minimisers]\label{thm:localmin_regularity}
  Let $F$ satisfy Hypotheses \ref{hyp:generalgrowth_F}, $\Omega \subset \mathbb R^n$ be a bounded domain and let $u \in \WW^{k,\varphi}(\Omega,\bb R^N)$ satisfy one of the following minimality properties:
  \begin{enumerate}[label=(\roman*)]
    \item\label{eq:local_case1} There is an $N$-function $\psi$ such that $u$ is a strong $\WW^{k,\psi}$-minimiser, and we have
      \begin{equation}\label{eq:psi_finiteness}
    \int_{\Omega'} \psi(\lambda \lvert\nabla^ku\rvert) \,\d x < \infty
      \end{equation} 
      for all $\lambda>0$ and $\Omega' \Subset\Omega$.

    \item\label{eq:local_case2} We have $u$ is a $\WW^{k,\infty}$-minimiser and moreover $u \in \WW^{k,\infty}_{\mathrm{loc}}(\Omega,\mathbb R^N)$ such that
  \begin{equation}\label{eq:linfty_smallness}
    \limsup_{r \to 0} \left(\esssup_{y \in \B_r(x)} \, \lvert \nabla^k u -(\nabla^k u)_{\B_r(x)}\rvert \right) < \delta
  \end{equation} 
  locally uniformly in $x \in \Omega$.
  \end{enumerate}
  Then for each $x_0 \in \Omega,$ $M>0$ and $\alpha \in (0,1),$ there exists $\eps>0$ and $R_0>0$ such that for any $R<R_0$ for which $\lvert(\nabla^ku)_{\B_R(x_0)}\rvert\leq M$ and
  \begin{equation}
    \mathrm E_M(x_0,R) = \dashint_{\B_R(x_0)} \varphi_{1+M}(\lvert\nabla^ku-(\nabla^ku)_{\B_R(x_0)}\rvert)\,\d x \leq \eps,
  \end{equation} 
  we have $u$ is of class $\CC^{k,\alpha}$ in $\B_{R/2}(x_0).$
\end{thm}

\begin{rem}
  If $\psi$ satisfies the $\Delta_2$ condition then (\ref{eq:psi_finiteness}) is equivalent to requiring that $u \in \WW^{k,\psi}_{\loc}(\Omega,\bb R^N),$ however extra care is needed in the absence of the $\Delta_2$-condition due the failure of suitable density results in $\LL^{\psi}.$ To the best of our knowledge is not known if this additional condition (\ref{eq:psi_finiteness}) is necessary, even in the polynomial case when $\psi = t^q.$

  In the $\WW^{k,\infty}$ case however, the construction in \cite[Section 7]{article:KristensenTaheri03} implies that it is insufficient to assume that $u$ is a $\WW^{k,\infty}$-local minimiser that lies in $\WW^{k,\infty}_{\loc}(\Omega,\bb R^N).$
\end{rem}

The key difference from the minimising setting is the lack of a full Caccioppoli inequality; to apply the minimising condition we need to work on balls of sufficiently small radii, which prevents us from applying the iteration argument used in the proof of Lemma \ref{lem:caccioppoli}. However one can still establish a weaker form which will suffice to complete the argument.

\begin{lem}[Caccioppoli-type inequality]\label{lem:localmin_caccioppoli}
  Assume the setup of Theorem \ref{thm:localmin_regularity}, and let $\kappa \in (0,1).$ Then there is $R_0 > 0$ and $\eps>0$ such that if $\B_r(x) \subset \B_{R_0}(x_0) \Subset \Omega$ for which $\lvert(\nabla^ku)_{\B_r(x)}\rvert \leq M$ and $\mathrm E_M(x,r) \leq \eps,$ then define $a_{x,r}$ such that $\nabla^{k-1}a_{x,r}$ satisfies Lemma \ref{lem:affine_approx} in $\B_r(x)$ with $\nabla^{k-1}u$ and that
  \begin{equation}\label{eq:particular_a}
    \dashint_{\B_r} \D^{\alpha}(u- a_{x,r}) \,\d y = 0
  \end{equation}  
  for all $\lvert\alpha \rvert \leq k-2.$ Then we have
  \begin{equation}
    \begin{split}
      &\dashint_{\B_{r/2}(x)} \varphi_{1+M}(\lvert \nabla^k u - \nabla^k a_{x,r}\rvert) \,\d y \\
      &\quad\leq \kappa\, \dashint_{\B_{r}(x)} \varphi_{1+M}(\lvert \nabla^k u - \nabla^k a_{x,r}\rvert) \,\d y 
      +C\,\dashint_{\B_{r}(x)} \varphi_{1+M}\left( \frac{\lvert u-a_{x,r}\rvert}{r^k} \right)  \,\d y.
    \end{split}
  \end{equation} 
\end{lem}

\begin{proof}
  Fix $\zeta \in (0,1)$ to be specified later, and let $0<t<s<r$ such that $r \leq \zeta(s-t).$ Then let $\eta \in \CC^{\infty}_c(\B_R(x_0))$ be a cutoff such that $\mathbbm{1}_{\B_t(x)} \leq \eta \leq \mathbbm{1}_{\B_s(x)}$ and $\lvert \nabla^j\eta\rvert \leq \frac{C}{(s-t)^{j}}$ for all $0 \leq j \leq k.$ We wish to use the minimising condition with $\xi = \eta(u-a_{x,r}),$ so we need to verify this is possible provided $R_0$ is sufficiently small.

  In the case of \ref{eq:local_case1}, we we will use the Poincar\'e inequality proved in \cite[Lemma 1]{article:BhattacharyaLeonetti91} which is valid for general $N$-functions $\psi$. We claim this allows us to estimate
  \begin{equation}
    \begin{split}
      \int_{\B_r(x)} \psi(\lvert \nabla \xi\rvert) \,\d y 
      &\leq \sum_{j=0}^k \int_{\B_r(x)} \psi\left( C\,\frac{\lvert \nabla^j (u-a_{x,r})\rvert}{(s-t)^{k-j}} \right) \,\d y \\
      &\leq C\sum_{j=0}^{k} \int_{\B_r(x)} \psi\left( \frac{Cr^{k-j}}{(s-t)^{k-j}}\lvert \nabla^{k} (u-a_{x,r})\rvert \right) \,\d y \\
      &\leq C \int_{\B_{R_0}(x_0)} \psi\left( C \zeta^{-k} (\lvert\nabla^ku\rvert+M) \right) \,\d y,
    \end{split}
  \end{equation} 
  using the additivity inequality $\psi(s+t) \leq \psi(2s)+\psi(2t)$ in the first line, followed by iteratively applying the Poincar\'e inequality using (\ref{eq:particular_a}) which applies for all $\lvert\alpha\rvert\leq k-1$ by choice of $\nabla^{k-1}a_{x,r}.$
  Then we can choose $R_0>0$ to be sufficiently small so this bound is smaller than $\delta$, thereby ensuring that $\xi$ is a valid test function.
    In the case of \ref{eq:local_case2} we can estimate
  \begin{equation}\label{eq:xi_infty_estimate}
    \begin{split}
      \norm{\nabla^k\xi}_{\LL^{\infty}(\Omega,\bb M_k)} 
      &\leq \norm{\nabla^ku - (\nabla^ku)_{\B_R(x_0)}}_{\LL^{\infty}(\B_r(x),\bb M_k)} +\lvert \nabla^ka_{x,r} - (\nabla^ku)_{\B_R(x_0)}\rvert \\
      &\quad+ C\sum_{j=0}^{k-1} \norm{\frac{\nabla^j(u-a_{x,r})}{(s-t)^{k-j}}}_{\LL^{\infty}(\B_r(x),\bb M_j)} \\
      &\leq \delta_1 + \lvert \nabla^ka_{x,r} - (\nabla^ku)_{\B_R(x_0)}\rvert \\
      &\quad+C \sum_{j=0}^{k-1} \zeta^{j-k} \norm{r^{k-j}(\nabla^j(u-a_{x,r}))}_{\LL^{\infty}(\B_r(x),\bb M_j)},
    \end{split}
  \end{equation} 
  where $R_0$ is chosen to be sufficiently small so that
  \begin{equation}
    \delta_1 :=
    \esssup_{y \in \B_{r}(x)} \, \lvert \nabla^k u -(\nabla^k u)_{\B_r(x)}\rvert < \delta,
  \end{equation} 
  for all $\B_r(x) \subset \B_R(x_0)$, using the additional condition (\ref{eq:linfty_smallness}). Now an analogous argument as in \cite[p.~76]{article:KristensenTaheri03} allows us to prove the pointwise estimate
  \begin{equation}\label{eq:uniform_lowerderivatives}
    \begin{split}
      &\norm{r^{k-j}(\nabla^j(u-a_{x,r}))}_{\LL^{\infty}(\B_r(x),\bb M_j)} \\
      &\quad\leq (2N)^{k-j} \norm{\nabla^k(u - a_{x,r})}_{\LL^{\infty}(\B_r(x),\bb M_k)},
    \end{split}
  \end{equation} 
  and we can also apply (\ref{eq:affine_derivative}) from Lemma \ref{lem:affine_approx} to estimate
  \begin{equation}\label{eq:a_avg_estimate}
    \lvert \nabla^ka_{x,r} - (\nabla^ku)_{\B_R(x_0)}\rvert \leq \dashint_{\B_r(x)} \lvert \nabla^k u - (\nabla^ku)_{\B_R(x)}\rvert \,\d x \leq \varphi_{1+M}^{-1}(\mathrm{E}_M(x,r)).
  \end{equation}
  From this we claim we have the interpolation estimate
  \begin{equation}\label{eq:interpolation_estimate_linfty}
    \begin{split}
      &\norm{r^{k-j}\nabla^j(u-a_{x,r})}_{\LL^{\infty}(\B_r(x),\bb M_j)} \\
      &\quad\leq \gamma \norm{\nabla^k(u - a_{x,r})}_{\LL^{\infty}(\B_r(x),\bb M_k)} + C_{\gamma}\dashint_{\B_r(x)} \lvert \nabla^k u - (\nabla^ku)_{\B_R(x)}\rvert \,\d x \\
      &\quad\leq \gamma \norm{\nabla^k(u - a_{x,r})}_{\LL^{\infty}(\B_r(x),\bb M_k)} + C_{\gamma}\varphi^{-1}_{1+M}(\mathrm{E}_M(x,r))
    \end{split}
  \end{equation} 
    for all $\gamma > 0$. 
    We will show the first inequality for general $u \in \WW^{k,\infty}(\B_r(x),\mathbb R^N)$ by means of a contradiction argument, noting the second line follows from Jensen's inequality.
    Rescaling to the case of the unit ball $\B = \B_1(0)$ we will assume the inequality does not hold, so there exists $(u_m)_m \subset \WW^{k,\infty}(\B,\mathbb R^N)$ and $0 \leq j \leq k-1$ for which for which
    \begin{equation}
      \begin{split}
        1 &= \lVert \nabla^j(u_m - a_{m}) \rVert_{\LL^{\infty}(\B,\mathbb M_j)} \\
          &> \gamma \lVert \nabla^k(u_m - a_{m}) \rVert_{\LL^{\infty}(\B,\mathbb M_k)} 
          + m\dashint_{\B} \lvert \nabla^k u_m - (\nabla^k u_m)_{\B}\rvert \,\d x,      
  \end{split}
  \end{equation} 
  holds for all $m$,
  where $a_{m}$ is defined as $a_{0,1}$ for $u_m$; in particular $(\nabla^i(u_m-a_m))_{\B_r(x)} = 0$ for all $0\leq i \leq k-1$ and \eqref{eq:a_avg_estimate} holds.
  Then by \eqref{eq:uniform_lowerderivatives} and Arzel\`a-Ascoli, passing to a subsequence there exists $v$ such that $\nabla^i(u_m - a_{x,r,m}) \to \nabla^iv$ uniformly on $\B$ for all $0 \leq i \leq k-1$, and since $\dashint_{\B} \lvert \nabla^ku_m-(\nabla^ku_m)_{\B}\rvert\,\d x \to 0$, combining with \eqref{eq:a_avg_estimate} we see that $(\nabla^kv)_{\B} = 0$. 
  Therefore since $\nabla^kv \equiv 0$ in $\B$ and $(\nabla^iv)_{\B} = 0$ for all $0 \leq i \leq N$, it follows that $v \equiv 0$ in $\B$, contradicting the fact that $\lVert \nabla^j(u_m - a_m)\rVert_{\LL^{\infty}(\B,\mathbb M_j)}=1$ for all $m$. Hence \eqref{eq:interpolation_estimate_linfty} holds.

    Therefore combining our estimates in \eqref{eq:xi_infty_estimate} we have
    \begin{equation}
      \lVert \nabla^k \xi \rVert_{\LL^{\infty}(\Omega,\mathbb M_k)} \leq \delta_1 + \gamma \lVert \nabla^k(u-a_{x,r})\rVert_{\LL^{\infty}(\B_R(x),\mathbb M_k)} + C_{\gamma} \varphi_{1+M}^{-1}(\eps),
    \end{equation} 
    which can be chosen to be smaller than $\delta$ by shrinking $\gamma, \eps>0$ as necessary.
    Hence, in both cases, the $\xi$ is a valid test function.

  Now we can argue as in Lemma \ref{lem:caccioppoli} from the minimising case; writing $\widetilde\varphi = \varphi_{1+M},$  noting that $w_{x,r} = u - a_{x,r}$ is $F_{\nabla^ka_{x,r}}$-extremal we can estimate
  \begin{equation}\label{eq:caccioppoli_preiterate}
    \int_{\B_t(x)} \widetilde\varphi\left(\lvert\nabla^kw_{x,r}\rvert\right) \,\d y \leq \theta\int_{\B_s(x)} \widetilde\varphi\left( \lvert\nabla^kw_{x,r}\rvert \right)\,\d y + C \int_{\B_r} \widetilde\varphi\left(\frac{\lvert w_{x,r}\rvert}{(s-t)^k}\right)\,\d y.
  \end{equation} 
  analogously to (\ref{eq:caccioppoli_middle1})--(\ref{eq:caccioppoli_middle3}), where $\theta>0$ is given and independent of $s,t.$ Then there is some $\ell\geq 1$ such that $\theta^\ell \leq \kappa,$ so we will iteratively apply this estimate $\ell$-times. 

  For this take $s_j = (1-\frac{j}{2\ell})r, t_j = (1-\frac{j-1}{2\ell})r$ for $1 \leq j \leq \ell,$ and note that $t_0 = r,$ $s_{\ell} = \frac{r}2,$ and $(t_j-s_j) =\frac{r}{2\ell}.$ Therefore taking $\zeta=\frac1{2\ell},$ we can apply (\ref{eq:caccioppoli_preiterate}) with $s_j,t_j$ and chain the inequalities to get
  \begin{equation}
    \int_{\B_{r/2}(x)} \widetilde\varphi\left(\lvert\nabla^kw_{x,r}\rvert\right) \,\d y \leq \kappa\int_{\B_r(x)} \widetilde\varphi\left( \lvert\nabla^kw_{x,r}\rvert \right) \,\d y + C \int_{\B_r} \widetilde\varphi\left(\frac{\lvert w_{x,r}\rvert}{R^k}\right)\,\d y,
  \end{equation} 
  as required.
\end{proof}

From here the proof of Theorem \ref{thm:localmin_regularity} proceeds analogously as in the minimising case. Note that both strong $\WW^{k,\psi}$ and $\WW^{k,\infty}$-local minimisers are $F$-extremal, so the harmonic approximation result (Lemma \ref{lem:harmonic_approximation}) can be applied to $u.$ We will sketch the remaining argument, pointing out the key modifications.

\begin{proof}[Proof of Theorem {\ref{thm:localmin_regularity}}]
  For $\B_r(x) \subset \B_R(x_0)$ we will consider the slightly modified excess energy
  \begin{equation}
  \widetilde{\mathrm E}_M(x,r) = \dashint_{\B_r(x)} \varphi_{1+M}\left( \lvert \nabla^k(u-a_{x,r})\rvert \right)  \,\d y,
  \end{equation} 
  with $a_{x,r}$ as in Lemma \ref{lem:localmin_caccioppoli}, which satisfies $\mathrm E_M(x,r) \sim \widetilde{\mathrm E}_M(x,r)$ by convexity of $\varphi$ and (\ref{eq:affine_derivative}). Then similarly as in Lemma \ref{lem:excess_decay}, we claim that if $\lvert\nabla^ka_{x,2\sigma r}\rvert,\lvert\nabla^ka_{x,r}\rvert \leq M$ and $\mathrm E_M(x,r)\leq 1,$ then for each $\kappa, \tilde\delta \in (0,1),$ $\sigma \in \left(0,\frac14\right)$ we have the excess decay estimate
  \begin{equation}
    \mathrm E_M(x,\sigma r) \leq C\left( \sigma + C_{\sigma}(\kappa+\tilde\delta) + C_{\sigma,\tilde\delta}\gamma_M(\mathrm E_M(x,r)) \right) \mathrm E_M(x,r),
  \end{equation} 
  where $\gamma_M$ is as in Lemma \ref{lem:harmonic_approximation}. Taking $a_1 = a_{x,r}$ we apply Lemma \ref{lem:harmonic_approximation} to obtain the unique solution $h \in u-a_1 + \WW^{k,\varphi}_0(\B_r(x),\bb R^N)$ to the problem
  \begin{equation}
    (-1)^k \nabla^k : F''(\nabla^k a_1)\nabla^kh = 0
  \end{equation} 
  in $\B_r(x).$ Then we define $a_2 = \sum_{\lvert\alpha\rvert\leq k} \frac{\D^{\alpha}h(x_0)}{\alpha!}(x-x_0)^{\alpha}$ as before, then we can apply Lemma \ref{lem:localmin_caccioppoli} to estimate
  \begin{equation}
    \begin{split}
      \widetilde E(x,\sigma r) &\leq \kappa \,\dashint_{\B_{2\sigma r}} \varphi_{1+M}\left( \lvert\nabla^k(u-a_{x,2\sigma r})\rvert \right) \,\d x \\
                               &\quad+ C\,\dashint_{\B_{2\sigma r(x)}} \varphi_{1+M}\left( \frac{\lvert u - a_{x,2\sigma r}\rvert}{(2\sigma r)^k} \right) \,\d x\\
                    &\leq C\kappa \, \dashint_{\B_{2\sigma r}} \varphi_{1+M}\left( \frac{\lvert\nabla^{k-1}(u-a_{x,r})\rvert}{2\sigma R} \right) \,\d x \\
                    & \quad + C\,\dashint_{\B_{2\sigma r(x)}} \varphi_{1+M}\left( \frac{\nabla^{k-1}(\lvert u - a_{x,2\sigma r}\rvert)}{2\sigma r} \right) \,\d x\\
                    &\leq C_{\sigma}\kappa \mathrm E_M(x,r) + C\,\dashint_{\B_{2\sigma r}} \varphi_{1+M}\left( \frac{\nabla^{k-1}\left(\lvert u - a_1 - a_2\rvert\right)}{2\sigma r} \right) \,\d x,
        \end{split}
  \end{equation} 
  where we have used (\ref{eq:affine_derivative2}) followed by the Poincar\'e inequality on $\B_r(x)$ for the first term, and the quasi-minima property (\ref{eq:affine_quasiminima}) of $\nabla^{k-1}a_{x,2\sigma r}.$ From here the rest follows by arguing exactly as in Lemma \ref{lem:excess_decay}.

  Given the above excess decay estimate, the theorem follows by an analogous iteration argument. A slight modification is needed to ensure that $\lvert\nabla^ka_{x,2\sigma r}\rvert,\lvert\nabla^ka_{x,r}\rvert \leq C(M),$ which can be done using (\ref{eq:affine_derivative}) and a similar argument as in \eqref{eq:average_differences}. We also choose $\kappa,\tilde\delta$ so that $CC_{\sigma}(\kappa+\tilde\delta) \leq \frac13\sigma^{2\alpha}$ and the rest follows as in the proof of Theorem \ref{thm:main_epsreg}.
\end{proof}

\section*{Acknowledgements}
The work was carried out while the author was a DPhil student at the University of Oxford, and the author would like to thank his advisor Jan Kristensen for the many helpful discussions and suggestions.


\begin{thebibliography}{99}
%
\bibitem{article:AcerbiFusco84}
E.\,Acerbi and N.\,Fusco, “Semicontinuity problems in the calculus of
variations,” \emph{Arch.\,Ration.\,Mech.\,Anal.}, vol.\,86, no.\,2,
pp.\,125–145, 1984, \textsc{issn}: 0003-9527. \textsc{doi}:
\doi{10.1007/BF00275731}.
%
\bibitem{article:AcerbiFusco87}
——, “A regularity theorem for minimizers of quasiconvex integrals,”
\emph{Arch.\,Ration.\,Mech.\,Anal.}, vol.\,99, no.\,3, pp.\,261–281, Sep.\,1987,
\textsc{issn}: 0003-9527. \textsc{doi}: \doi{10.1007/BF00284509}.
%
\bibitem{article:AcerbiMingione01}
E.\,Acerbi and G.\,Mingione, “Regularity results for a class of
quasiconvex functionals with nonstandard growth,” \emph{Ann.\,Della\,Scuola\,Norm.\,Super.\,Pisa - Cl.\,Sci.}, vol.\,30, no.\,2, pp.\,311–339, 2001.
%
\bibitem{book:AdamsFournier03}
R.\,A.\,Adams and J.\,Fournier, \emph{Sobolev Spaces}. Academic
Press, 2003, p.\,305, \textsc{isbn}: 978-0-08-054129-7.
%
\bibitem{article:ADN2}
S.\,Agmon, A.\,Douglis, and L.\,Nirenberg, “Estimates near the boundary
for solutions of elliptic partial differential equations satisfying general
boundary conditions II,” \emph{Commun.\,Pure Appl.\,Math.}, vol.\,17,
no.\,1, pp.\,35–92, Feb.\,1964, \textsc{issn}: 00103640. \textsc{doi}:
\doi{10.1002/cpa.3160170104}.
%
\bibitem{article:BalciEtAl20}
A.\,K.\,Balci, L.\,Diening, and M.\,Weimar, “Higher order
Calderón-Zygmund estimates for the $p$-Laplace equation,” \emph{J.\,Diff\,Eq.}, vol.\,268, no.\,2, pp.\,590–635, Jan.\,2020,
\textsc{issn}: 0022-0396. \textsc{doi}: \doi{10.1016/j.jde.2019.08.009}.
%
\bibitem{article:BarlinEtAl24}
M.\,B\"arlin, F.\,Gmeineder, C.\,Irving, J.\,Kristensen, “$\mathcal{A}$-harmonic approximation and partial regularity, revisited,” Dec.\,2022. arXiv:\arxiv{2212.12821} \texttt{[math.AP]}
%
\bibitem{article:BhattacharyaLeonetti91}
T.\,Bhattacharya and F.\,Leonetti, “A new poincaré inequality and its
application to the regularity of minimizers of integral functionals with
nonstandard growth,” \emph{Nonlinear Anal.\,Theory.\,Methods.\,App}., vol.\,17, no.\,9, pp.\,833–839, Jan.\,1991, \textsc{issn}:
0362546X. \textsc{doi}: \doi{10.1016/0362-546X(91)90157-V}.
%
\bibitem{article:Bogelein12}
V.\,Bögelein, “Partial regularity for minimizers of discontinuous
quasi-convex integrals with degeneracy,” \emph{J.\,Diff.\,Eq.}, vol.\,252, no.\,2, pp.\,1052–1100, Jan.\,2012, \textsc{issn}:
0022-0396. \textsc{doi}: \doi{10.1016/j.jde.2011.09.031}.
%
\bibitem{article:Cagnetti11}
F.\,Cagnetti, “$k$-quasi-convexity reduces to quasi-convexity,”
\emph{Proc.\,R.\,Soc.\,Edinb.\,Sect.\,Math.}, vol.\,141, no.\,4, pp.\,673–708,
Aug.\,2011, \textsc{issn}: 1473-7124, 0308-2105. \textsc{doi}: \doi{10.1017/S0308210510000867}.
%
\bibitem{article:CamposCordero17}
J.\,Campos Cordero, “Boundary regularity and sufficient conditions for
strong local minimizers,” \emph{J.\,Funct.\,Anal.}, vol.\,272, no.\,11,
pp.\,4513–4587, Jun.\,2017, \textsc{issn}: 00221236. \textsc{doi}: \doi{10.1016/j.jfa.2017.02.027}.
%
\bibitem{article:CamposCordero21}
——, “Partial regularity for local minimizers of variational integrals
with lower order terms,” \emph{Q.\,J.\,Math.},
2021, \textsc{issn}: 0033-5606, 1464-3847. \textsc{doi}:
\doi{10.1093/qmath/haab056}.
%
\bibitem{article:CarozzaFuscoMingione}
M.\,Carozza, N.\,Fusco, and G.\,Mingione, “Partial regularity of
minimizers of quasiconvex integrals with subquadratic growth,”
\emph{Ann.\,Mat.\,Pura.\,Ed.\,Appl.}, vol.\,175, no.\,1, pp.\,141–164, Dec.\,1998,
\textsc{issn}: 0373-3114.\,\textsc{doi}: \doi{10.1007/BF01783679}.
%
\bibitem{article:CarozzaNapoli03}
M.\,Carozza and A.\,P.\,di Napoli, “Partial regularity of local minimizers
of quasiconvex integrals with sub-quadratic growth,” \emph{Proc.\,R.\,Soc.
Edinb.\,Sect.\,Math.}, vol.\,133, no.\,6, pp.\,1249–1262, Dec.\,2003,
\textsc{issn}: 1473-7124, 0308-2105. \textsc{doi}:
\doi{10.1017/S0308210500002900}.
%
\bibitem{article:ChlebickaEtAl19}
I.\,Chlebicka, F.\,Giannetti, and A.\,Zatorska-Goldstein, “Elliptic
problems with growth in nonreflexive Orlicz spaces and with measure or
$L^1$ data,” \emph{J.\,Math.\,Anal.\,App.},
vol.\,479, no.\,1, pp.\,185–213, Nov.\,2019, \textsc{issn}: 0022-247X.
\textsc{doi}: \doi{10.1016/j.jmaa.2019.06.022}.
%
\bibitem{article:Cianchi96}
A.\,Cianchi, “A Sharp Embedding Theorem for Orlicz-Sobolev Spaces,”
\emph{Indiana Univ.\,Math.\,J.}, vol.\,45, no.\,1, pp.\,39–65, 1996,
\textsc{issn}: 0022-2518.
%
\bibitem{article:DalMasoEtAl04}
G.\,Dal Maso, I.\,Fonseca, G.\,Leoni, and M.\,Morini, “Higher-Order
Quasiconvexity Reduces to Quasiconvexity,” \emph{Arch.\,Ration.\,Mech.\,Anal.}, vol.\,171, no.\,1, pp.\,55–81, Jan.\,2004, \textsc{issn}: 1432-0673.
\textsc{doi}: \doi{10.1007/s00205-003-0278-1}.
%
\bibitem{article:DeFilippis21}
C.\,De Filippis, “Quasiconvexity and partial regularity via nonlinear
potentials,”  
\emph{J.\,Math.\,Pures.\,Appl.}, vol.\,163, pp.\,11-82, Jul.\,2022. \textsc{doi}: \doi{10.1016/j.matpur.2022.05.001}

\bibitem{article:DeFilippisStroffolini2023}
C.\,De Filippis and B.\,Stroffolini, “Singular multiple integrals and nonlinear potentials,” \emph{J.\,Func.\,Anal.}, vol.\,285, no.\,2, Jul.\,2023. \textsc{doi}: \doi{10.1016/j.jfa.2023.109952}.

%
\bibitem{article:degiorgi68}
E.\,De Giorgi, “Un esempio di estremali discontinue per un problema
variazionale di tipo ellittico,” \emph{Boll.\,Dell.\,Unione Mat.\,Ital.}, vol.\,4,
pp.\,135–137, 1968.
%
\bibitem{article:DieningEttwein08}
L.\,Diening and F.\,Ettwein, “Fractional estimates for non-differentiable
elliptic systems with general growth,” vol.\,20, no.\,3, pp.\,523–556, May
2008, \textsc{issn}: 1435-5337. \textsc{doi}: \doi{10.1515/FORUM.2008.027}.
%
\bibitem{article:Dieningetal12}
L.\,Diening, D.\,Lengeler, B.\,Stroffolini, and A.\,Verde, “Partial
Regularity for Minimizers of Quasi-convex Functionals with General
Growth,” \emph{SIAM Jounral Math.\,Anal.}, vol.\,44, pp.\,3594–3616,
2012. \textsc{doi}: \doi{10.1137/120870554}.
%
\bibitem{article:DRS10}
L.\,Diening, M.\,Ružička, and K.\,Schumacher, “A decomposition
technique for John domains,” \emph{Ann.\,Acad.\,Sci.\,Fenn.\,Math.},
vol.\,35, pp.\,87–114, Mar.\,2010, \textsc{issn}: 1239629X. \textsc{doi}:
\doi{10.5186/aasfm.2010.3506}.
%
\bibitem{article:DieningEtAl12a}
L.\,Diening, B.\,Stroffolini, and A.\,Verde, “The $\varphi$-harmonic
approximation and the regularity of $\varphi$-harmonic maps,” \emph{J.\,Diff.\,Eq.}, vol.\,253, no.\,7, pp.\,1943–1958, Oct.\,2012,
\textsc{issn}: 0022-0396. \textsc{doi}: \doi{10.1016/j.jde.2012.06.010}.
%
\bibitem{article:DongKim18}
H.\,Dong and D.\,Kim, “On $L^p$-estimates for elliptic and parabolic
equations with $A_p$ weights,” \emph{Trans.\,Am.\,Math.\,Soc.}, vol.\,370,
no.\,7, pp.\,5081–5130, Feb.\,2018, \textsc{issn}: 0002-9947. \textsc{doi}: \doi{10.1090/tran/7161}.
%
\bibitem{article:DuzaarGrotowski00}
F.\,Duzaar and J.\,F.\,Grotowski, “Optimal interior partial regularity for
nonlinear elliptic systems: The method of A-harmonic approximation,”
\emph{manuscripta math.}, vol.\,103, no.\,3, pp.\,267–298, Nov.\,2000,
\textsc{issn}: 1432-1785. \textsc{doi}: \doi{10.1007/s002290070007}.
%
\bibitem{article:DuzaarMingione04}
F.\,Duzaar and G.\,Mingione, “The p-harmonic approximation and the
regularity of p-harmonic maps,” \emph{Cal.\,Var}, vol.\,20, no.\,3,
pp.\,235–256, Jul.\,2004, \textsc{issn}: 1432-0835. \textsc{doi}: \doi{10.1007/s00526-003-0233-x}.
%
\bibitem{article:DuzaarSteffen02}
F.\,Duzaar and K.\,Steffen, “Optimal interior and boundary regularity
for almost minimizers to elliptic variational integrals,” \emph{J.\,Für
Reine Angew.\,Math.\,Crelles J.}, vol.\,2002, no.\,546, Jan.\,2002,
\textsc{issn}: 0075-4102, 1435-5345. \textsc{doi}: \doi{10.1515/crll.2002.046}.
%
\bibitem{article:EvansGariepy87}
L.\,C.\,Evans and R.\,F.\,Gariepy, “Blowup, Compactness and Partial
Regularity in the Calculus of Variations,” \emph{Indiana Univ.\,Math.\,J.}, vol.\,36, no.\,2, pp.\,361–371, 1987, \textsc{issn}: 0022-2518.
%
\bibitem{article:Evans86}
L.\,C.\,Evans, “Quasiconvexity and partial regularity in the calculus of
variations,” \emph{Arch.\,Ration.\,Mech.\,Anal.}, vol.\,95, no.\,3,
pp.\,227–252, Sep.\,1986, \textsc{issn}: 0003-9527. \textsc{doi}: \doi{10.1007/BF00251360}.
%
\bibitem{article:Franceschini19}
F.\,Franceschini, “Partial regularity for $\BV^{\mathscr{B}}$ local
minimizers,” M.S.\,thesis, Universit‘a degli Studi di Pisa, 2019.
%
\bibitem{article:FrieseckeEtAl02}
G.\,Friesecke, R.\,D.\,James, and S.\,Müller, “A theorem on geometric
rigidity and the derivation of nonlinear plate theory from
three-dimensional elasticity,” \emph{Commun.\,Pure Appl.\,Math.},
vol.\,55, no.\,11, pp.\,1461–1506, 2002, \textsc{issn}: 1097-0312.
\textsc{doi}: \doi{10.1002/cpa.10048}.
%
\bibitem{article:FuscoHutchinson85}
N.\,Fusco and J.\,Hutchinson, “$C^{1,\alpha}$ Partial regularity of functions
minimising quasiconvex integrals,” \emph{Manuscripta Math.}, vol.\,54,
no.\,1, pp.\,121–143, Mar.\,1985, \textsc{issn}: 1432-1785. \textsc{doi}: \doi{10.1007/BF01171703}.
%
\bibitem{article:GiaquintaModica86}
M.\,Giaquinta and G.\,Modica, “Partial regularity of minimizers of
quasiconvex integrals,” \emph{Ann.\,Inst.\,Henri Poincaré C.\,Anal.\,Non Linéaire}, vol.\,3, no.\,3,
pp.\,185–208, 1986.
%
\bibitem{book:GilbargTrudinger98}
D.~Gilbert and N.~Trudinger.
\newblock \emph{Elliptic Partial Differential Equations of Second Order}.
\newblock {Springer-Verlag Berlin Heidelberg}, 1998.
%
\bibitem{book:Giusti03}
E.\,Giusti, “Direct methods in the calculus of variations,” World Scientific, Jan.
2003, \textsc{isbn}: 978-981-238-043-2.\,\textsc{doi}: \doi{10.1142/5002}.
%
\bibitem{article:GiustiMiranda68}
E.\,Giusti and M.\,Miranda, “Sulla regolarità delle soluzioni deboli di
una classe di sistemi ellittici quasi-lineari,” \emph{Arch.\,Ration.\,Mech.\,Anal.}, vol.\,31, no.\,3, pp.\,173–184, Jan.\,1968, \textsc{issn}: 0003-9527.
\textsc{doi}: \doi{10.1007/BF00282679}.
%
\bibitem{article:Gmeineder21}
F.\,Gmeineder, “Partial regularity for symmetric quasiconvex
functionals on BD,” \emph{J.\,de Math.\,Pures et App.}, vol.\,145, pp.\,83–129, Jan.\,2021, \textsc{issn}: 0021-7824.
\textsc{doi}: \doi{10.1016/j.matpur.2020.09.005}.
%
\bibitem{article:GmeinederKristensen19}
F.\,Gmeineder and J.\,Kristensen, “Partial regularity for BV
minimizers,” \emph{Arch.\,Ration.\,Mech.\,Anal.}, vol.\,232, no.\,3,
pp.\,1429–1473, Jun.\,2019, \textsc{issn}: 0003-9527. \textsc{doi}: \doi{10.1007/s00205-018-01346-5}.
%
\bibitem{article:GmeinederKristensen19a}
——, “Regularity for higher order quasiconvex problems with linear
growth from below,” Mar.\,2019. arXiv:\arxiv{1903.08124} \texttt{[math.AP]}
%
\bibitem{article:GmeinederKristensen24}
——, “Quasiconvex functionals of $(p,q)$-growth and the partial regularity of relaxed minimisers,” \emph{Arch.\,Ration.\,Mech.\,Anal.}, vol.\,248, no.\,80, 2024. \textsc{doi}: \doi{10.1007/s00205-024-02013-8}
%
\bibitem{article:Guidorzi00}
M.\,Guidorzi, “A Remark on Partial Regularity of Minimizers of
Quasiconvex Integrals of Higher Order,” \emph{Univ.\,Degli Studi Trieste
Dipartimento Sci.\,Mat.}, no.\,33, pp.\,1–24, 2000.
%
\bibitem{article:Habermann13}
J.\,Habermann, “Partial regularity for nonlinear elliptic systems with
continuous growth exponent,” \emph{Annali di Matematica}, vol.\,192,
no.\,3, pp.\,475–527, Jun.\,2013, \textsc{issn}: 1618-1891. \textsc{doi}: \doi{10.1007/s10231-011-0233-y}.
%
\bibitem{article:Irving21}
C.\,Irving, “$\mathrm{BMO}$ $\varepsilon$-regularity results for
solutions to Legendre-Hadamard elliptic systems,” \emph{Calc.\,Var.\,Partial\, Differ.\,Eq.}, vol.\,62, no.\,166, Jun.\,2023. \textsc{doi}: \doi{10.1007/s00526-023-02492-9}
%
\bibitem{thesis:Irving22}
——, “Regularity theory in the calculus of variations and elliptic systems,” DPhil Thesis, University of Oxford, 2022.
%
\bibitem{book:KokilashviliKrbec91}
V.\,Kokilashvili and M.\,Krbec, \emph{Weighted Inequalities in Lorentz
and Orlicz Spaces}. World Scientific, Dec. 1991, \textsc{isbn}:
978-981-02-0612-3. \textsc{doi}: \doi{10.1142/1367}.
%
\bibitem{article:KrasnoselskiiRutickii61}
M.\,A.\,Krasnosel’skii and Y.\,B.\,Rutickii, \emph{Convex {{Functions}}
and {{Orlicz Spaces}}}. P.\,Noordhoff, 1961.
%
\bibitem{article:KristensenTaheri03}
J.\,Kristensen and A.\,Taheri, “Partial regularity of strong local
minimizers in the multi-dimensional calculus of variations,” \emph{Arch.\,Ration.\,Mech.\,Anal.}, vol.\,170, no.\,1, pp.\,63–89, Nov.\,2003, \textsc{issn}:
0003-9527. \textsc{doi}: \doi{10.1007/s00205-003-0275-4}.
%
\bibitem{article:Kronz02}
M.\,Kronz, “Partial regularity results for minimizers of quasiconvex
functionals of higher order,” \emph{Ann.\,Inst.\,Henri Poincaré C.\,Anal.\,Non Linéaire}, vol.\,19, no.\,1, pp.\,81–112, 2002, \textsc{issn}: 02941449.
\textsc{doi}: \doi{10.1016/S0294-1449(01)00072-5}.
%
\bibitem{article:Krylov07}
N.\,V.\,Krylov, “Parabolic and Elliptic Equations with VMO
Coefficients,” \emph{Commun.\,Partial Differ.\,Equ.}, vol.\,32, no.\,3,
pp.\,453–475, Mar.\,2007, \textsc{issn}: 0360-5302. \textsc{doi}: \doi{10.1080/03605300600781626}.
%
%
\bibitem{book:Maligranda89}
L.\,Maligranda, “Orlicz spaces and interpolation.” Seminários de Matemática, 1989,  \textsc{issn}: 0103-5258 ; 5
%
\bibitem{article:Mazya69}
V.\,G.\,Maz’ya, “Examples of nonregular solutions of quasilinear elliptic
equations with analytic coefficients,” \emph{Funct\,Anal\,Appl}, vol.\,2,
no.\,3, pp.\,230–234, 1969, \textsc{issn}: 0016-2663, 1573-8485.
\textsc{doi}: \doi{10.1007/BF01076124}.
%
\bibitem{article:Meyers65}
N.\,G.\,Meyers, “Quasi-convexity and lower semi-continuity of multiple
variational integrals of any order,” \emph{Trans.\,Amer.\,Math.\,Soc.},
vol.\,119, pp.\,125–149, 1965, \textsc{issn}: 0002-9947. \textsc{doi}: \doi{10.2307/1994235}.
%
\bibitem{article:MooneySavin16}
C.\,Mooney and O.\,Savin, “Some singular minimizers in low dimensions
in the calculus of variations,” \emph{Arch.\,Ration.\,Mech.\,Anal.},
vol.\,221, no.\,1, pp.\,1–22, Jul.\,2016, \textsc{issn}: 0003-9527. \textsc{doi}: \doi{10.1007/s00205-015-0955-x}.
%
\bibitem{article:Morrey52}
C.\,B.\,Morrey, “Quasi-convexity and the lower semicontinuity of
multiple integrals.,” \emph{Pacific J.\,Math.}, vol.\,2, no.\,1, pp.\,25–53,
1952, \textsc{issn}: 0030-8730.
%
\bibitem{article:Morrey68}
——, “Partial regularity results for non-linear elliptic systems,” Tech.
Rep.\,7, 1968, pp.\,649–670.
%
\bibitem{article:MullerSverak03}
S.\,Müller and V.\,Šverák, “Convex integration for Lipschitz mappings
and counterexamples to regularity,” \emph{Ann.\,Math.}, vol.\,157, no.\,3,
pp.\,715–742, May 2003, \textsc{issn}: 0003-486X. \textsc{doi}: \doi{10.4007/annals.2003.157.715}.
%
\bibitem{article:Necas77}
J.\,Nečas, “Example of an irregular solution to a nonliear elliptic system
with analytic coefficients and conditions for regularity,” in \emph{Theory
Nonlinear Oper.\,{{Constr}}.\,{{Asp}}.}, R.\,Kluge and W.\,Müller, Eds., Berlin:
Akademie-Verlag, 1977, pp.\,197–206.
%
\bibitem{article:RaoRen91}
M.\,Rao and Z.\,Ren, “Theory of Orlicz spaces,” Basel, 1991.
%
\bibitem{article:Riordan55}
W.\,Riordan, “On the interpolation of operations,” PhD Thesis, University of Chicago, 1995.
\bibitem{article:Schemm11}
S.\,Schemm, “Partial regularity of minimizers of higher order integrals
with $(p, q)$-growth,” \emph{ESAIM Control Optim.\,Calc.\,Var.}, vol.\,17,
no.\,2, pp.\,472–492, Apr.\,2011, \textsc{issn}: 1292-8119, 1262-3377.
\textsc{doi}: \doi{10.1051/cocv/2010016}.
%
\bibitem{article:SchemmSchmidt09}
S.\,Schemm and T.\,Schmidt, “Partial regularity of strong local
minimizers of quasiconvex integrals with $(p, q)$-growth,” \emph{Proc.\,R.
Soc.\,Edinb.\,Sect.\,Math.}, vol.\,139, no.\,3, pp.\,595–621, Oct.\,2009,
\textsc{issn}: 1473-7124, 0308-2105. \textsc{doi}: \doi{10.1017/S0308210507001278}.
%
\bibitem{book:Simon83}
L.\,Simon, \emph{Lectures on Geometric Measure Theory}. Centre for
Mathematical Analysis, Australian National University, 1983, p.\,272,
\textsc{isbn}: 0-86784-429-9.
%
\bibitem{book:Simon96}
——, \emph{Theorems on {{Regularity}} and {{Singularity}} of
{{Energy Minimizing Maps}}}. Basel: Birkhäuser Basel, 1996,
\textsc{isbn}: 978-3-7643-5397-1 978-3-0348-9193-6. \textsc{doi}: \doi{10.1007/978-3-0348-9193-6}.
%
\bibitem{article:Stampacchia65}
G.\,Stampacchia, “The spaces $\mathcal L^{p,\lambda},$ $N^{p,\lambda}$ and
interpolation,” \emph{Ann.\,Della.\,Sc.\,Norm Super Pisa Cl.\,Sci}, vol.\,19,
no.\,3, pp.\,443–462, 1965.
%
\bibitem{article:SverakYan02}
V.\,Sverák and X.\,Yan, “Non-Lipschitz minimizers of smooth uniformly
convex functionals,” \emph{Proc.\,Natl.\,Acad.\,Sci.\,USA}, vol.\,99, no.\,24,
pp.\,15 269–15 276, Nov.\,2002, \textsc{issn}: 0027-8424. \textsc{doi}: \doi{10.1073/pnas.222494699}.
%
\bibitem{article:Szekelyhidi04}
L.\,Székelyhidi, “The regularity of critical points of polyconvex
functionals,y \emph{Arch.\,Ration.\,Mech.\,Anal.}, vol.\,172, no.\,1,
pp.\,133–152, Apr.\,2004, \textsc{issn}: 0003-9527. \textsc{doi}: \doi{10.1007/s00205-003-0300-7}.
%
\bibitem{book:Taylor1_11}
M.\,E.\,Taylor, \emph{Partial Differential Equations {{I}}},
ser. Applied Mathematical Sciences.\,New York, NY: Springer New York,
2011, vol.\,115, \textsc{isbn}: 978-1-4419-7054-1. \textsc{doi}: \doi{10.1007/978-1-4419-7055-8}.
%
\bibitem{book:Ziemer89}
W.\,P.\,Ziemer, \emph{Weakly {{Differentiable Functions}}},
ser. Graduate Texts in Mathematics. New York, NY: Springer New York,
1989, vol. 120, \textsc{isbn}: 978-1-4612-6985-4 978-1-4612-1015-3.
\textsc{doi}: \doi{10.1007/978-1-4612-1015-3}.
%
\bibitem{article:Zygmund56}
A.\,Zygmund, “On a theorem of Marcinkiewicz concerning interpolation of operators,“ \emph{J.\,Math.\,Pures Appl.}, vol.\,35, pp.\,223--248, 1956. \textsc{doi}: \doi{10.1007/978-94-009-1045-4_12}
\end{thebibliography}
\end{document}